\documentclass[reqno]{amsart}

\usepackage{xypic,hyperref,todonotes,comment}
\usepackage{graphics,amssymb}
\usepackage{pagecolor,lipsum}
\usepackage{latexsym}
\usepackage{mathrsfs}
\xyoption{curve}
\usepackage{hyperref}
\hypersetup{ 
colorlinks,
citecolor=blue,
filecolor=blue,
linkcolor=black,
urlcolor=green
}

\newtheorem{lemma}{Lemma}[section]
\newtheorem{proposition}[lemma]{Proposition}
\newtheorem{theorem}[lemma]{Theorem}
\newtheorem{corollary}[lemma]{Corollary}

\newtheorem{conjecture}[lemma]{Conjecture}

\theoremstyle{definition}

\newtheorem{definition}[lemma]{Definition}
\newtheorem{remark}[lemma]{Remark}

\newcommand{\mfk}[1]{\mathfrak{#1}}
\newcommand{\mbb}[1]{\mathbb{#1}}
\newcommand{\mcl}[1]{\mathcal{#1}}
\newcommand{\mrm}[1]{\mathrm{#1}}
\newcommand{\msc}[1]{\mathscr{#1}}
\newcommand{\mbf}[1]{\mathbf{#1}}

\DeclareMathOperator{\Hom}{Hom}
\DeclareMathOperator{\End}{End}

\DeclareMathOperator{\Aut}{Aut}
\DeclareMathOperator{\Fun}{Fun}
\DeclareMathOperator{\rep}{rep}
\DeclareMathOperator{\Rep}{Rep}
\DeclareMathOperator{\corep}{corep}
\DeclareMathOperator{\Corep}{Corep}
\DeclareMathOperator{\Ind}{Ind}
\DeclareMathOperator{\Spec}{Spec}
\DeclareMathOperator{\Fr}{Fr}
\DeclareMathOperator{\SL}{SL}
\DeclareMathOperator{\PSL}{PSL}
\DeclareMathOperator{\Sp}{Sp}
\DeclareMathOperator{\ord}{ord}
\DeclareMathOperator{\can}{can}
\DeclareMathOperator{\M}{M}

\newcommand{\ot}{\otimes}

\newcommand{\dotU}{\dot{\mathbf{U}}}
\newcommand{\hatU}{\widehat{\mathbf{U}}}
\newcommand{\dotu}{\dot{\mathbf{u}}}
\newcommand{\hatu}{\widehat{\mathbf{u}}}
\newcommand{\doto}{\dot{\mathfrak{o}}}
\renewcommand{\1}{\mathbf{1}}
\renewcommand{\O}{\mathscr{O}}
\renewcommand{\binom}[2]{{\Small\left[\begin{matrix}\ #1\ \\ #2 \end{matrix} \right]}}

\title[]{Log-modular quantum groups at even roots of unity and the quantum Frobenius I}
\date{\today}

\author{Cris Negron}
\thanks{This work was supported by NSF Postdoctoral Research Fellowship DMS-1503147}
\email{cnegron@email.unc.edu}
\address{Department of Mathematics, University of North Carolina, Chapel Hill, NC 27599}

\begin{document}
\maketitle

\begin{abstract}
We construct log-modular quantum groups at even order roots of unity, both as finite-dimensional ribbon quasi-Hopf algebras and as finite ribbon tensor categories, via a de-equivariantization procedure.  The existence of such quantum groups had been predicted by certain conformal field theory considerations, but constructions had not appeared until recently.  We show that our quantum groups can be identified with those of Creutzig-Gainutdinov-Runkel in type $A_1$, and Gainutdinov-Lentner-Ohrmann in arbitrary Dynkin type.  We discuss conjectural relations with vertex operator algebras at $(1,p)$-central charge.  For example, we explain how one can (conjecturally) employ known linear equivalences between the triplet vertex algebra and quantum $\mfk{sl}_2$, in conjunction with a natural $\PSL_2$-action on quantum $\mfk{sl}_2$ provided by our de-equivariantization construction, in order to deduce linear equivalences between ``extended" quantum groups, the singlet vertex operator algebra, and the $(1,p)$-Virasoro logarithmic minimal model.  We assume some restrictions on the order of our root of unity outside of type $A_1$, which we intend to eliminate in a subsequent paper.
\end{abstract}

\section{Introduction}

This paper concerns the production of certain non-semisimple ``non-degenerate" quantum groups at even order roots of unity.  In order to highlight the issues we mean to address in this work, let us consider the case of quantum $\mfk{sl}_2$.
\par

We have the standard small quantum group, or quantum Frobenius kernel, $u_q(\mfk{sl}_2)$ in Lusztig's divided power algebra $U_q(\mfk{sl}_2)$~\cite{lusztig90,lusztig90II}, i.e.\ the Hopf subalgebra generated by $E$, $F$, and $K$.  It has been shown that, at arbitrary even order $q$, the Hopf algebra $u_q(\mfk{sl}_2)$ admits no quasitriangular structure~\cite{kondosaito11,gainutdinovrunkel17}.  This is in contrast to the odd order case, where the small quantum group is always quasitriangular.  Indeed, this quasitriangular property holds, in a certain sense, at all parameters {\it except} for even order roots of unity.
\par

From another perspective, it is known that there is a linear equivalence between representations of the small quantum group $u_q(\mfk{sl}_2)$ and representations of a certain strongly-finite vertex operator algebra--the triplet VOA~\cite{kausch91,fjelstadfuchshwangsemikhatovtipunin02,gainutdinovetal06,adamovicmilas08,nagatomotsuchiya11}.  Hence $\rep u_q(\mfk{sl}_2)$ apparently admits \emph{some} braided tensor structure, via the logarithmic tensor theory of Huang, Lepowsky, and Zhang~\cite{hlz14,hlz} (cf.~\cite[Conjecture 5.7]{gainutdinovrunkel19}).  So, one may conclude that there is some error in the definition of the Hopf structure on quantum $\mfk{sl}_2$ at an even order root of unity which, after it has been remedied, will reproduce the CFT-inspired tensor structure as the natural tensor structure on $\rep u_q(\mfk{sl}_2)$ induced by the coproduct on $u_q(\mfk{sl}_2)$ (see e.g.~\cite{gainutdinovetal06,feigintupin,gainutdinovrunkel17,creutzigetal}).
\par

This slippage between representation theory and conformal field theory is not unique to type $A_1$, although the corresponding conformal field theories are not well-developed outside of type $A_1$.  One expects, in the conclusion, that there is an appropriate correction to the definition of the small quantum group $u_q(\mfk{g})$, for an arbitrary simple Lie algebra $\mfk{g}$ over $\mbb{C}$ and even order $q$, under which the category $\rep u_q(\mfk{g})$ is braided, and even log-modular (cf.~\cite[Conjecture 3.2]{adamovicmilas14}).  To be clear about our terminology:

\begin{definition}[\cite{creutziggannon17}]
A log-modular tensor category $\msc{C}$ is a finite, non-degenerate, ribbon tensor category.
\end{definition}

One could refer to such categories simply as modular tensor (as opposed to fusion) categories, although we would like to draw a distinction between our quantum group categories and those of, say,~\cite{andersenparadowski95,rowell06}.  By non-degenerate we mean that $\msc{C}$ is braided and maximally non-symmetric, in the precise sense of Definition~\ref{def:ndg} below.
\par

In the present work we examine the issues discussed above from a representation theoretic, and tensor categorical, perspective.  In particular, we clarify how one can correct the apparent ``singular" behaviors of quantum groups at even order roots of unity by employing representation theoretic techniques.  We discuss the relevance of our findings from a conformal field theory perspective in Section~\ref{sect:(1,p)intro} below, and discuss other recent constructions of log-modular quantum groups in Section~\ref{sect:CFTintro}.
\par

Let us consider an almost simple algebraic group $G$, over $\mbb{C}$, and the associated category of quantum group representations
\[
\rep G_q=\left\{\begin{array}{c}
\text{Finite-dimensional representations of Lusztig's divided power}\\
\text{algebra $U_q(\mfk{g})$ which are graded by the character lattice $X$ of }G
\end{array}\right\}.
\]
In the above expression $\mfk{g}$ is the Lie algebra of $G$, and $q$ is always an {\it even} order root of unity.  The category $\rep G_q$ admits a canonical ribbon (braided) structure, and Lusztig's quantum Frobenius yields a braided tensor embedding $\Fr:\rep G^\vee\to \rep G_q$ which has M\"uger central image, where $G^\vee$ is a specific almost simple dual group to $G$ (see Section~\ref{sect:qfrob}).

We focus in the introduction on the simply-connected case, as results become sporadic away from the weight lattice.  However, in the body of the text we deal with arbitrary almost simple $G$.

\begin{theorem}[{\ref{cor:ff},\ref{lem:ribbon},\ref{cor:nondegen}}]
Let $G$ be simply-connected and suppose that the character lattice for $G$ is strongly admissible at (even order) $q$.  Then the de-equivariantization 
\[
(\rep G_q)_{G^\vee}:=\left\{\begin{array}{c}
\text{\rm Finitely presented $\Fr\O(G^\vee)$-modules in }\rep G_q
\end{array}\right\}
\]
has the canonical structure of a finite, non-degenerate, ribbon tensor category.  That is to say, $(\rep G_q)_{G^\vee}$ is a log-modular tensor category.
\end{theorem}

We note that outside of the simply-connected setting the de-equivariantization $(\rep G_q)_{G^\vee}$ may fail to be ribbon, although it is always finite and non-degenerate.  We explain our ``strongly admissible" condition in detail below.  Let us say for now that $\SL_2$ has strongly admissible character lattice at arbitrary $q$, and that outside of type $A_1$ this basically means that $4$ divides the order of $q$.  (See Section~\ref{sect:admissible}.)  We call $(\rep G_q)_{G^\vee}$ the log-modular quantum Frobenius kernel for $\rep G_q$, at even order $q$, or simply the {\it log-modular kernel}.
\par

From the perspective of this work, the de-equivariantization $(\rep G_q)_{G^\vee}$ is {\it the} canonical form for the small quantum group at even order $q$.  However, we show at Proposition~\ref{prop:qhuM_exists} that $(\rep G_q)_{G^\vee}$ admits an algebraic incarnation as the representation category of a ribbon quasi-Hopf algebra $u^{\mrm{M}}_q(G)$.  As a consequence of Proposition~\ref{prop:qhuM_exists} below, and non-degeneracy of the de-equivariantization, we find that $u^{\M}_q(G)$ is in fact log-modular.
\par

We describe the quasi-Hopf algebras $u^{\M}_q(G)$ in detail in Section~\ref{sect:uM}.  The formula for the comultiplication in particular is given in Lemma~\ref{lem:468}.  To identify with the above discussion, one should take the simply-connected form $u_q^{\M}(G_\mrm{sc})$ specifically as the error-corrected version of $u_q(\mfk{g})$.
\par

The $u^{\M}_q(G)$ arrive to us as subalgebras in (a completion of) the corresponding divided power algebra $U_q(G)$.  It is precisely the subalgebra generated by the elements $\mathsf{E}_\alpha:=K_\alpha E_\alpha$ and $F_\alpha$, and the character group $Z^\vee$ for the quotient $Z$ of the weight lattice by the $\ord(q)/2$-scaling of the root lattice.  For the standard nilpotent subalgebras $u^+_q,\ u^-_q\subset U_q(G)$, we provide in Lemma~\ref{lem:triangle} a triangular decomposition
\[
u^-_q\ot\mbb{C}[Z^\vee]\ot u^+_q\overset{\cong}\to u^{\M}_q(G).
\]

The quasi-Hopf structure on $u^{\M}_q(G)$ is not canonical, but depends on a choice of function $\omega:X\times X\to \mbb{C}^\times$, which essentially quantifies the failure of the algebra $\Fr\O(G^\vee)$ to be central in the quantum function algebra $\O_q(G)$.  We call $\omega$ a balancing function, and its precise properties are described in Section~\ref{sect:balancing}.  At the categorical level, however, the tensor structure on $\rep u^{\M}_q(G)$ is unique up to isomorphism, via the identification with the canonical form $(\rep G_q)_{G^\vee}$.

\begin{theorem}[{\S \ref{sect:uM}, \ref{prop:qhuM_exists}}]\label{thm:158}
Let $G$ be simply-connected with strongly admissible character lattice at (even order) $q$.  There is a log-modular quasi-Hopf algebra $u^{\M}_q(G)$ which admits a ribbon equivalence
\[
fib^\omega:(\rep G_q)_{G^\vee}\overset{\sim}\to \rep u^{\M}_q(G).
\]
The comultiplication and $R$-matrix for $u^{\M}_q(G)$ depend on a choice of balancing function $\omega$ for $G$, but are unique up to braided tensor equivalence.  The ribbon element for $u^{\M}_q(G)$ is independent of the choice of balancing function.
\end{theorem}

For $\mfk{sl}_2$, for example, the dual group to $\SL_2$ is $\SL_2^\vee=\PSL_2$.  In this case one finds that $u_q^{\M}(\SL_2)$ is in fact the standard small quantum group $u_q(\SL_2)\subset U_q(\SL_2)$, with some alternate choice of quasi-Hopf structure induced by its identification with the categorical kernel $(\rep(\SL_2)_q)_{\PSL_2}$.  We discuss this example in Section~\ref{sect:sl2}.
\par

We note that Theorem~\ref{thm:158} was obtained at the $\mbb{C}$-linear level, i.e.\ as a $\mbb{C}$-linear equivalence, in earlier work of Arkhipov and Gaitsgory~\cite{arkhipovgaitsgory03}.  In particular, the definition of the {\it algebra} $u^{\M}_q(G)$ was observed already in~\cite{arkhipovgaitsgory03} (see also~\cite[\S 3.11]{andersenparadowski95}).

\subsection{Identifications with the log-modular quantum groups of Creutzig et al.~\cite{creutzigetal} and Gainutdinov et al.~\cite{gainutdinovlentnerohrmann}}\label{sect:CFTintro}

Independent of the present paper, constructions of log-modular quantum groups at even order roots of unity have appeared in work of Creutzig, Gainutdinov, and Runkel~\cite{gainutdinovrunkel17,creutzigetal}, in type $A_1$, and in work of Gainutdinov, Lentner, and Ohrmann~\cite{gainutdinovlentnerohrmann} in arbitrary Dynkin type.
\par

In~\cite{creutzigetal} a quasi-Hopf algebra $u^\phi_q(\mfk{sl}_2)$ was produced via a local module construction.  The local module construction of~\cite{creutzigetal} is motivated by certain CFT considerations and, from our perspective, is essentially a de-equivariantization (see Section~\ref{sect:id}).  We note that the results of~\cite{creutzigetal} followed earlier work of Gainutdinov and Runkel~\cite{gainutdinovrunkel17} in which the authors produced the quasi-Hopf algebra $u^\phi_i(\mfk{sl}_2)$ for $\mfk{sl}_2$ at parameter $q=i$, essentially by hand.
\par

In~\cite{gainutdinovlentnerohrmann} the authors proceed via an Andruskiewitch-Schneider like approach (cf.\ \cite{andruskiewitschschneider10,AAsurvey}), where the quantum groups $u_q(G)$ are produced as quotients of Drinfeld doubles of Nichols algebras $B(V)$, with $V$ an object in the braided category of representations of a cocycle perturbed group algebra.  So, $V$ lives in a braided category which does not admit a fiber functor in general, and the construction of $B(V)$ takes place in this category as well.
\par

As remarked in~\cite{gainutdinovlentnerohrmann}, all of the constructions of quantum groups from~\cite{gainutdinovrunkel17,creutzigetal,gainutdinovlentnerohrmann} agree, when appropriate.  We prove in Section~\ref{sect:id} that our quantum groups $u^{\M}_q(G)$ agree with those of Creutzig, Gainutdinov, Runkel~\cite{gainutdinovrunkel17,creutzigetal} and Gainutdinov, Lentner, Ohrmann~\cite{gainutdinovlentnerohrmann}, at the ribbon categorical level.

\begin{remark}
In addition to the production of certain small quantum groups, much of the labors of~\cite{gainutdinovrunkel17,creutzigetal,lentner,gainutdinovlentnerohrmann} are directed towards producing and refining relationships between quantum groups and vertex operator algebras/CFTs.
\end{remark}

\begin{remark}
One point which is consistent across all of the references discussed above, as well as the present work, is that the failure of the na\"ive quantum group $u_q(\mfk{g})$ to admit an $R$-matrix, in general, has to do with some defect in the Cartan part $\mbb{C}[Z^\vee]$.  So, the na\"ive quantum group and (what we call) the log-modular quantum group only differ due to some alteration in the Cartan part.
\end{remark}

\subsection{Relevance for the ``logarithmic Kazhdan-Lusztig equivalence" at $(1,p)$-central charge}\label{sect:(1,p)intro}
Take $u^{\M}_q(\mfk{sl}_2)$ the simply-connected form $u^{\M}_q(\SL_2)$.  We discuss here the situation in type $A_1$, and fix $q$ of order $2p$.  Some aspects of the story in arbitrary Dynkin type are recalled in the concluding paragraphs.
\par

As we alluded to earlier, there is a conjectured equivalence of ribbon tensor categories
\[
f_p:\rep u_q^{\M}(\mfk{sl}_2)\overset{\sim}\to \rep \mcl{W}_p,
\]
where $\mcl{W}_p$ is the triplet vertex operator algebra~\cite{kausch91,fuchsetal04,adamovicmilas08}.  This conjecture was first proposed in the paper~\cite{gainutdinovetal06}, and it has been shown that such an equivalence $f_p$ exists at the level of {\it $\mbb{C}$-linear} categories~\cite{gainutdinovetal06,nagatomotsuchiya11}.  (So, without the tensor product.)  It is conjectured that the equivalence $f_p$ for the triplet algebra lifts to additional equivalences
\[
\rep_{wt} u^H_q(\mfk{sl}_2)\overset{\sim}\to \rep_{\langle s\rangle} \mcl{M}_p,\ \ \rep G_q\overset{\sim}\to \rep \mcl{LM}(1,p),
\]
where $u^H_q(\mfk{sl}_2)$ is the so-called unrolled quantum group, $\mcl{M}_p$ is the singlet VOA, and $\rep \mcl{LM}(1,p)$ is a certain subcategory of the representations of the $(1,p)$-Virasoro which we leave unspecified for the moment~\cite{bgt12,creutzigmilas14,cgp15}.  (See Section~\ref{sect:(1,p)}.)
\par

Here we are concerned with means of obtaining equivalences for the singlet and Virasoro from the known additive equivalence $f_p$ for the triplet algebra.  As we argue in Section~\ref{sect:(1,p)}, this problem may be approached via considerations of certain natural $\PSL_2$ actions on $\rep u^{\M}_q(\mfk{sl}_2)$ and $\rep \mcl{W}_p$.  The action of $\PSL_2$ on $\rep \mcl{W}_p$ is well-established in the CFT literature~\cite{adamoviclinmilas13}, while the action on $\rep u^{\M}_q(\mfk{sl}_2)$ is deduced from our construction of the log-modular quantum group as a $\PSL_2$ de-equivariantization of $\rep(\SL_2)_q$.

\begin{conjecture}[\ref{conj:1064}]
The linear equivalence $f_p:\rep u^{\M}_q(\mfk{sl}_2)\overset{\sim}\to \rep \mcl{W}_p$ is $\PSL_2$-equivariant, or can be chosen to be $\PSL_2$-equivariant.
\end{conjecture}

A positive solution to Conjecture~\ref{conj:1064} would produce explicit functors
\[
\operatorname{A}:\rep_\mbb{Z}u^H_q(\mfk{sl}_2)\to \rep \mcl{M}_p,\ \ \operatorname{B}:\rep (\SL_2)_q\to \rep \mcl{LM}(1,p)
\]
via the triplet equivalence $f_p$.
\par

Let us conclude with a short discussion of the situation in other Dynkin types.  We again take $u^{\M}_q(\mfk{g})=u^{\M}_q(G_{sc})$ the simply-connected form.  Analogs $\mcl{W}_p(\mfk{g})$ of the triplet algebra in arbitrary Dynkin type were introduced in work of Feigin and Tipunin~\cite{feigintupin}, with the triplet $\mcl{W}_p=\mcl{W}_p(\mfk{sl}_2)$ recovered in type $A_1$.  These vertex operator algebras are conjectured to be strongly finite~\cite{adamovicmilas14}--and in particular $C_2$-cofinite--although outside of types $A_1$ this conjecture remains completely open.  One can see~\cite{flandolilentner18} for a specific discussion of type $B$.
\par

Supposing strong finiteness of the algebras $\mcl{W}_p(\mfk{g})$, it is additionally conjectured that there is an equivalence of braided tensor categories $\rep u_q^{\M}(\mfk{g})\to \rep\mcl{W}_p(\mfk{g})$~\cite{lentner,gainutdinovlentnerohrmann}.  Lentner proposed~\cite[Conjecture 6.8 \& 6.9]{lentner} that the dual group $G^\vee$ acts naturally on $\mcl{W}_p(\mfk{g})$ so that the invariants $\mcl{W}_p(\mfk{g})^{G^\vee}$ are the associated $W$-algebra $\msc{W}_k(\mfk{g})$ \cite{feiginfrenkel90} at a corresponding level $k$.  Although we have clearly stacked up quite a few conjectures at this point, we would suggest that the proposed $G^\vee$ action on $\mcl{W}_p(\mfk{g})$ should correspond to our action of $G^\vee$ on $\rep u_q^{\M}(\mfk{g})$, and that the representations of the big quantum group $\rep G_q$ should be identified with a distinguished tensor subcategory in $\rep\msc{W}_k(\mfk{g})$, just as in the type $A_1$ case.

\subsection*{Acknowledgements}
This work has benefited from numerous conversations with Pavel Etingof, Azat Gainutdinov, Dennis Gaitsgory, Simon Lentner, and Ingo Runkel.  Section~\ref{sect:remPSL2} was developed in conversation with Etingof.  I thank Runkel and Gainutdinov, and also Ehud Meir, for their hospitality during a visit to Universit\"{a}t Hamburg.  Thanks also to the referees for many helpful comments, recommendations, and astute observations regarding quantum groups at low order parameters.

\tableofcontents

\section{Preliminaries}

All algebraic structures (algebras, schemes, algebraic groups, categories, etc.) are over $\mbb{C}$.  An algebraic group is an affine group scheme of finite type over $\mbb{C}$.  The standing conditions for this document are that {\it $q$ is a root of unity of even order $2l$, with $l$ positive,} and that {\it $G$ is an almost simple algebraic group with strongly admissible character lattice at $q$} (defined in Section~\ref{sect:admissible} below).
\par

For any algebra $A$, we let $\rep A$ denote the category of finite-dimensional $A$-modules.  We let $\Rep A$ denote the category of $A$-modules which are the union of their finite-dimensional submodules.  We adopt a similar notation $\corep A$ and $\Corep A$ for comodules over a coalgebra, but note that $\Corep A$ happens to be equal to the category of arbitrary comodules here.  For a $\mbb{C}$-linear category $\msc{C}$ we let $\Ind\msc{C}$ denote the corresponding $\Ind$-category, i.e.\ the completion of $\msc{C}$ with respect to filtered colimits, so that $\Ind(\rep A)=\Rep A$ (resp. $\Ind(\corep A)=\Corep A$) for example.

\subsection{Basics on (braided) tensor categories}

We refer the reader to~\cite{egno15}, and in particular \cite[\S 4.1 \& \S 8.1]{egno15}, for basics on tensor categories.  Concisely, a tensor category (over $\mbb{C}$) is a $\mbb{C}$-linear, abelian monoidal category which has duals, has a simple unit object $\1$, and satisfies certain local finiteness conditions.  Following~\cite{etingofostrik04}, we call a tensor category $\msc{C}$ \emph{finite} if $\msc{C}$ has finitely many simples and enough projectives.  This implies that $\msc{C}$ is equivalent to the representation category of a finite-dimensional algebra, as a $\mbb{C}$-linear abelian category.
\par

A tensor functor between tensor categories is an exact $\mbb{C}$-linear monoidal functor.  A fiber functor for a tensor category $\msc{C}$ is a faithful tensor functor to $Vect$, $F:\msc{C}\to Vect$.  By an \emph{embedding} $F:\msc{D}\to \msc{C}$ of tensor categories we mean a fully faithful tensor functor for which $F(\msc{D})$ is closed under taking subobjects in $\msc{C}$.  When $\msc{D}$ is a finite tensor category this subobject closure property is a consequence of fully faithfulness~\cite[\S 6.3]{egno15}.  In the infinite setting there are fully faithful tensor functors which are not embeddings.
\par

A braided tensor category is a tensor category $\msc{C}$ equipped with a family of natural isomorphisms $c_{V,W}:V\ot W\to W\ot V$, at all $V$ and $W$ in $\msc{C}$, which satisfies the braid relations~\cite[Definition 8.1.1]{egno15}.  A braided tensor functor $F:\msc{C}\to \msc{D}$ is a tensor functor which respects the braiding, in the sense that braidings from $\msc{C}$ and $\msc{D}$ induce the same maps $F(V)\ot F(W)\to F(W)\ot F(V)$.  We write $c^2_{V,W}$ for the double braiding $c_{W,V}c_{V,W}:V\ot W\to V\ot W$.

\begin{definition}\label{def:ndg}
The M\"uger center of a braided tensor category $\msc{C}$ is the full tensor subcategory of $\msc{C}$ consisting of all objects $V$ for which the double braiding transformation $c^2_{V,-}:V\ot -\to V\ot -$ is the identity.  We call a braided tensor category $\msc{C}$ non-degenerate if its M\"uger center is trivial, i.e.\ if any M\"uger central $V$ is isomorphic to a sum of the unit object $V\cong \1^{\oplus r}$.
\end{definition}

When $\msc{C}$ is finite, our definition of non-degeneracy, in terms of the M\"uger center, is equivalent to all other reasonable notions of non-degeneracy~\cite{shimizu19}.
\par

We recall that a symmetric tensor category is one for which the double braiding $c^2_{-,-}$ is the identity, globally, and a \emph{Tannakian} category is a braided tensor category $\msc{C}$ which admits a braided fiber functor to $Vect$.  Note that a Tannakian category must be symmetric, although not all symmetric tensor categories are Tannakian.  (For example, the category $sVect$ of super vector spaces is non-Tannakian, as it has objects with self-braiding $-id_{V\ot V}$.)

\begin{definition}
A ribbon structure on a braided tensor category $\msc{C}$ is a choice of a family of natural endomorphisms $\theta_V:V\to V$ which satisfy $(\theta_V)^\ast=\theta_{V^\ast}$ and $\theta_{V\ot W}=(\theta_V\ot \theta_W)c^2_{V,W}$, for all $V$ and $W$.
\end{definition}

\subsection{Almost simple algebraic groups}

Let $G$ be an almost simple algebraic group over $\mbb{C}$, with root lattice $Q$ and weight lattice $P$.  Recall that $G$ is specified, up to isomorphism, by its Lie algebra $\mfk{g}=\mrm{Lie}(G)$ and choice of character lattice $X$ between $Q$ and $P$.  The character lattice appears abstractly as the group of maps from a maximal torus $T\subset G$ to $\mbb{G}_m$, $X=\Hom_\mrm{AlgGrp}(T,\mbb{G}_m)$.  (By $\mbb{G}_m$ we mean the multiplicative group $\mbb{C}^\ast$ with its standard algebraic group structure.)  For $G$ of adjoint type we have $X=Q$, and for $G$ simply-connected $X=P$.
\par

We let $\Delta=\{\alpha_1,\dots,\alpha_n\}$ denote the simple roots in $X$, and $\Phi\subset X$ denote the collection of all roots.  For each simple $\alpha_i$ we have an associated integer $d_i=d_{\alpha_i}\in \{1,2,3\}$ and diagonal matrix $D=\mrm{diag}\{d_1,\dots, d_n\}$ for which $D[a_{ij}]$ is symmetric, where the $a_{ij}$ are the Cartan integers for $G$.

We have the Cartan pairing $\langle\ ,\ \rangle:Q\times Q\to \mbb{Z}$, defined by the Cartan integers $\langle \alpha_i,\alpha_j\rangle=a_{ij}$.  If we take $r$ to be the group exponent of the quotient $X/Q$, then this form extends to a unique $\mbb{Z}[\frac{1}{r}]$-valued form on $X$.  We have a unique symmetrization $(\ ,\ ):X\times X\to \mbb{Z}[\frac{1}{r}]$ of the Cartan form on $X$ defined by
\[
(\alpha_i,\alpha_j)=d_i\langle a_i,a_j\rangle=d_ia_{ij}.
\]
We call this symmetrized form the {\it(normalized) Killing form} on $X$, since the induced form on the complexification $X_\mbb{C}$ is identified with the standard Killing form on the dual $\mfk{h}^\ast$ of the Cartan subalgebra $\mfk{h}$ in $\mfk{g}$, up to scaling.

\begin{remark}
Note that the Cartan integer conventions for Lusztig~\cite{lusztig90II,lusztig93} are transposed relative to those of, say, Humphreys~\cite{humphreys12}.  We follow Lusztig's convention here, in order to produce a consistency between our presentation and the works of Lusztig, so that $\langle a_i,a_j\rangle=2(a_i,a_j)/(a_i,a_i)$ \cite[Definition 2.2.1]{lusztig93}.
\end{remark}

\subsection{Exponentiation of the Killing form on $X$}

Take again $r$ to be the exponent of the quotient $X/Q$, so that the Killing form on $X$ takes values in $\mbb{Z}[\frac{1}{r}]$.  For $q$ an arbitrary root of unity in $\mbb{C}$, with argument $\theta$, we may take the $r$-th root $\sqrt[r]{q}=\exp(2\pi i\theta/r)$.  We exponentiate the Killing form to arrive at the multiplicative form
\[
\Omega:X\times X\to \mbb{C}^\ast,\ \ \Omega(x,y):=(\sqrt[r]{q})^{r(x,y)}.
\]
Since $r(x,y)$ is an integer this form is well-defined.  Having established this point, we abuse notation throughout and write simply $\Omega(x,y)=q^{(x,y)}$.

\subsection{Representations of the quantum group $\rep G_q$ and the divided power algebra $U_q(\mfk{g})$}
\label{sect:repGq}

Take $q$ a root of unity of order $2l$, let $\mfk{g}$ be a simple Lie algebra over $\mbb{C}$, and for each root $\gamma\in\Phi$ take
\[
l_\gamma:=\text{the minimal positive integer such that }d_\gamma l_\gamma\in l\mbb{N},
\]
where $d_\gamma$ is the relative length $\lvert \gamma\lvert^2/\lvert\text{short root}\lvert^2$.  Following \cite[Chapter 35]{lusztig93}, \emph{we assume additionally that $l_\alpha>-\langle \alpha,\beta\rangle$ at all pairs of distinct simple roots $\alpha$, $\beta$}.  This condition is always satisfied in the simply-laced case, provided $l$ is positive, but requires that $l$ is not \emph{too} small outside of the simply-laced case.  

\begin{remark}
One can require that the comparison $l_\alpha>-\langle \alpha,\beta\rangle$ holds only at those $\alpha$ for which $l_\alpha>1$.  However, in applying this relaxation one should alter the definition of $u^M_q(G)$ (Section \ref{sect:uM}) in accordance with \cite[\S 35.4.1]{lusztig93}.
\end{remark}

Let $U_q=U_q(\mfk{g})$ be Lusztig's divided power quantum group specialized at $q$~\cite{lusztig90,lusztig90II}, with standard generators 
\[
E_\alpha,\ F_\alpha,\ K_\alpha,\ E_\alpha^{(l_\alpha)},\ F_\alpha^{(l_\alpha)},\ \binom{K_\alpha;0}{l_\alpha},\ \ \text{for all }\alpha\in \Delta.
\]
Here the $K_\alpha$ are grouplike, the $E_\alpha$ are $(K_\alpha,1)$-skew primitive, and the $F_\alpha$ are $(1,K_\alpha^{-1})$-skew primitive.  We let $\rep G_q$ denote the tensor subcategory in $\rep U_q(\mfk{g})$ consisting of objects $V$ such that:
\begin{enumerate}
\item[(a)] $V$ comes equipped with a grading by the character lattice, $V=\oplus_{\lambda\in X} V_\lambda$,
\item[(b)] For $v\in V_\lambda$ the torus elements in $U_q$ act by the corresponding eigenvalues,
\[
K_\alpha\cdot v=q^{(\alpha,\lambda)}v\ \text{and}\ \binom{K_\alpha; 0}{l_\alpha}\cdot v=\binom{\langle \alpha,\lambda\rangle}{l_\alpha}_{d_\alpha}v,\ \text{where $\binom{a}{b}_{d_\alpha}$ is the $q^{d_\alpha}$-binomial.
}
\]
\end{enumerate}
Morphisms in $\rep G_q$ are $U_q$-linear maps which preserve the $X$-grading.  (Obviously, $U_q=U_q(\mfk{g})$ here.)  For the materials of Section~\ref{sect:(1,p)}, we would like to understand the nature of $\rep G_q$ as a subcategory in $\rep U_q$.

\begin{proposition}\label{prop:273}
The faithful tensor functor $\rep G_q\to \rep U_q$ is a tensor embedding.
\end{proposition}

The proof of the proposition will follow from Lemma~\ref{lem:297} below.

\begin{remark}
The analogous map $\rep G_q\to \rep U_q$ is an equivalence at simply-connected $G$ when $q$ is of \emph{odd order}.  At even order $q$ the functor of Proposition~\ref{prop:273} is not essentially surjective for $G=\SL_2$ (see Section~\ref{sect:(1,p)2}), and thus not an equivalence, and we expect that it is not an equivalence for any $G$ at such $q$.
\end{remark}

For simple $\alpha$ let $f_\alpha\in P$ denote the corresponding fundamental weight in $P$, so that $(f_\alpha,\beta)=d_\beta\delta_{\alpha,\beta}$ at simple $\beta$.  Since $X\subset P$, we may write any element in $X$ uniquely as a linear combination of these $f_\alpha$, with coefficients in $\mbb{Z}$.

Consider $V$ in $\rep G_q$, and take a homogenous nonzero element $v\in V$.  For simple $\alpha\in \Delta$ consider the unique integer $0\leq m'_v(\alpha)<\ord(q^{d_\alpha})$ so that $K_\alpha v=q^{d_\alpha m'_v(\alpha)}v$ and take
\[
m_v(\alpha)=\left\{\begin{array}{ll}
m'_v(\alpha) & \text{if $\ord(q^{d_\alpha})$ is odd or $m'_v(\alpha)<\frac{\ord(q^{d_\alpha})}{2}$}\\
m'_v(\alpha)-\frac{\ord(q^{d_\alpha})}{2} & \text{else}.
\end{array}\right.
\]
Let also $n_v'(\alpha)\in\mbb{Z}$ be such that $v$ lies in the $n_v'(\alpha)$-eigenspace for the action of $\binom{K_\alpha;0}{l_\alpha}$ (cf.~\cite[Corollary 3.3]{lusztig89}) and take 
\[
n_v(\alpha)=\left\{
\begin{array}{ll}
n'_v(\alpha) & \text{if $\ord(q^{d_\alpha})$ is odd}\\
(-1)^{l(n_v'(\alpha)-1)}n_v'(\alpha) &\text{else}.
\end{array}\right.
\]
Finally, define $\ell_\alpha=\mrm{ord}(q^{d_\alpha})$ if the order of $q^{d_\alpha}$ is odd and $\mrm{ord}(q^{d_\alpha})/2$ otherwise

\begin{lemma}\label{lem:297}
Consider homogenous $v\in V$, for $V$ in $\rep G_q$, and take $m_v(\alpha),\ n_v(\alpha)\in\mbb{Z}$ as above.  Then the $X$-degree of $v$ is given by the formula
\begin{equation}\label{eq:deg}
\deg(v)=\sum_{\alpha\in\Delta}(m_v(\alpha)+(-1)^{m_v(\alpha)(\ord(q^{d_\alpha})-1)}\ell_\alpha n_v(\alpha))f_\alpha.
\end{equation}
\end{lemma}

\begin{proof}
We may assume $G$ is simply-connected, by way of the embedding from $\rep G_q$ to the simply-connected form.  Via the restriction functors $F_\alpha:\rep G_q\to \rep (\SL_2)_{q^{d_\alpha}}$ along the Hopf embeddings $U_{q^{d_\alpha}}(\mfk{sl}_2)\to U_q(\mfk{g})$, which sends $E$, $F$, and $K$ to $E_\alpha$, $F_\alpha$ and $K_\alpha$, it suffices to consider the case $G=\SL_2$.  Here the weight lattice is generated by the single fundamental weight $f=\frac{1}{2}\alpha$.  We note that $\ord(q^{d_\alpha})$ may be odd, in which case $\ord(q^{d_\alpha})=\ell_\alpha$, and make the analogous $\ell_\alpha$-demands as above in the definition of $\rep(\SL_2)_{q^{d_\alpha}}$.  In any case, we take $G=\SL_2$ and allow $q$ to be of possibly odd order.
\par

Take $v\in V$ of degree $cf$, for $V$ in $\rep(\SL_2)_q$, and assume first that $q$ is of even order $2\ell$.  Then we have
\[
\binom{K;0}{\ell}v=\binom{\langle \alpha,cf\rangle }{\ell}v=\binom{c}{\ell}v,
\]
and by definition $n'_v=\binom{c}{\ell}$.
We have directly that $\binom{r}{\ell}=0$ when $0\leq r<\ell$ and $\binom{\ell}{\ell}=1$, and also the general property
\[
\binom{k\ell+a}{\ell}=q^{-a\ell}\binom{k\ell}{\ell}+q^{k\ell^2}\binom{a}{\ell}=(-1)^{a}\binom{k\ell}{\ell}+(-1)^{k\ell}\binom{a}{\ell}
\]
(see~\cite[\S 1.3]{lusztig93}).  This gives $\binom{k\ell}{\ell}=(-1)^{\ell(k-1)}k$ by induction and $\binom{k\ell+r}{\ell}=(-1)^r(-1)^{\ell(k-1)}k$ for $0\leq r<\ell$.  So, in total,
\[
n'_v=\binom{c}{\ell}=(-1)^{c-\ell\lfloor\frac{c}{\ell}\rfloor}(-1)^{\ell(n'_v-1)}\lfloor\frac{c}{\ell}\rfloor.
\]
The difference $c-\ell\lfloor\frac{c}{\ell}\rfloor$ is $m_v$, since $Kv=q^cv$.  Hence
\[
\begin{array}{rl}
c & =c-\ell\lfloor\frac{c}{\ell}\rfloor+\lfloor\frac{c}{\ell}\rfloor\\
 &= m_v+(-1)^{m_v}(-1)^{\ell(n'_v-1)}\ell n_v'= m_v+(-1)^{m_v}\ell n_v.
\end{array}
\]
So we see $\deg(v)=cf=(m_v+(-1)^{m_v}\ell n_v)f$, as claimed.
\par

A similar, but easier, analysis yields the result for $\rep (\SL_2)_q$ when $q$ is of odd order.
\end{proof}

\begin{proof}[Proof of Proposition~\ref{prop:273}]
One sees from Lemma~\ref{lem:297} that the $X$-grading on $V$ in $\rep G_q$ is completely recoverable from the action of the torus elements in $U_q$.  Whence we find that morphisms $V\to W$ in $\rep U_q$ between $X$-graded objects preserve the $X$-grading, implying full faithfulness of the inclusion.  Furthermore, for a $v\in V$ in $X$-graded $V$ we may expand $v$ in terms of the grading $v=\sum_\lambda v_\lambda$ and, by Lemma~\ref{lem:297} we may take for any $\lambda\in X$ a torus element $t_\lambda\in U_q$ so that $t_\lambda v=v_\lambda$.  Hence any subobject $V'\subset V$ in $\rep U_q$ is $X$-graded as well.  Whence the inclusion $\rep G_q\to \rep U_q$ is an embedding.
\end{proof}

\subsection{The $R$-matrix for $\rep G_q$}
\label{sect:R}

Let $q$ be a root of unity of order $2l$, as before.  Recall our notation $\Omega:X\times X\to \mbb{C}^\times$ for the $q$-exponentiated Killing form.  According to~\cite[Chapter 32]{lusztig93} the category $\rep G_q$ is braided by the operator
\[
R=R^+\Omega^{-1}=(\sum_{n:\Phi^+\to \mbb{Z}_{\geq 0}}c_n(q) E_{\gamma_1}^{(n_1)}\dots E_{\gamma_w}^{(n_w)}\ot F_{\gamma_1}^{(n_1)}\dots F_{\gamma_w}^{(n_w)})\Omega^{-1},
\]
where the $c_n(q)$ are polynomials in $q^{\pm 1}$ with integer coefficients, $\{\gamma_1,\dots,\gamma_w\}$ is a normal ordering of the positive roots, and up to first order we have
\[
R=\left(1-(\sum_{\alpha\in \Delta}(q-q^{-1})E_\alpha\ot F_\alpha)+\dots\right)\Omega^{-1}.
\]
This linear term actually specifies $R$ entirely.  The corresponding braiding on $\rep G_q$ is given by
\[
\begin{array}{c}
c_{V,W}:V\ot W\to W\ot V,\\
\begin{array}{l}
c_{V,W}(v\ot w)=\operatorname{swap}(R\cdot v\ot w)\\
=q^{-(\deg v,\deg w)}\operatorname{swap}\left(\sum_{n:\Phi^+\to \mbb{Z}_{\geq 0}}c_n(q) E_{\gamma_1}^{(n_1)}\dots E_{\gamma_w}^{(n_w)}v\ot F_{\gamma_1}^{(n_1)}\dots F_{\gamma_w}^{(n_w)}w\right),
\end{array}\end{array}
\]
where $\operatorname{swap}$ is the standard vector space symmetry, and $v$ and $w$ are taken to be homogeneous in the above expression.  This braiding operation is well-defined as any object in $\rep G_q$ is annihilated by sufficiently high powers of any $E_\gamma$, $F_\gamma$.

\begin{remark}
In~\cite{lusztig93}, Lusztig's ``$R$-matrix" $R'$ is the reverse of our $R$-matrix, $R'=R_{21}$.  This is because the braiding employed in~\cite{lusztig93} is $R'\circ \operatorname{swap}$, which is equal to $\operatorname{swap}\circ R$.  We follow the convention of~\cite{egno15} with regards to $R$-matrices and braidings.
\end{remark}

The following result is well-known, and we omit a formal proof.

\begin{lemma}[{cf.~\cite[\S 8.3C]{charipressley95}}]\label{lem:367}
The coefficients $c_n(q)$ in the expression of the $R$-matrix are such that $c_n(q)=0$ whenever $n_\gamma\geq l_\gamma$ for any $\gamma\in \Phi^+$.
\end{lemma}

Lemma~\ref{lem:367} says that the $R$-matrix lives in a certain ``torus extended small quantum group" for $G$ at $q$ (denoted $\widehat{\mbf{u}}_q$ below).

\subsection{Algebras of global operators}
\label{sect:global_op}

\begin{definition}\label{def:global_op}
Let $\msc{C}$ be a locally finite $\mbb{C}$-linear category with fixed fiber functor $F:\msc{C}\to Vect$.  The algebra of global operators for $\msc{C}$ is the endomorphism ring $\End_{\Fun/\mbb{C}}(F)$.  For $\rep G_q$, we let $\hatU_q(G)$ denote the associated algebra of global operators (calculated with respect to the forgetful functor to $Vect$).
\end{definition}

By $\End_{\Fun/\mbb{C}}(F)$ we mean the algebra of natural endomorphisms of the $\mbb{C}$-linear functor $F$.  Elements of this algebra are families of linear maps $a_V:FV\to FV$, defined at all $V$ in $\msc{C}$, which satisfy $F(t)a_V=a_WF(t)$ for any map $t:V\to W$ in $\msc{C}$.  In this subsection we expand upon the the construction of the algebra $\hatU_q(G)$ for the quantum group.  We explain, in particular, that the algebra $\hatU_q(G)$ is identified with the completion of a familiar quantum group along a cofiltered system of ideals.
\par

For $\rep G_q$ we have Lusztig's modified algebra $\dotU_q(G)=\bigoplus_{\lambda\in X}U_q 1_\lambda$~\cite[Section 1.2]{kashiwara94} (see also \cite[Chapter 23 \& 31]{lusztig93}), which has $\rep\dotU_q(G)=\rep G_q$.  To be clear, $U_q1_\lambda$ is the cyclic module
\[
U_q(\mfk{g})/\left(\sum_\alpha U_q (K_\alpha-q^{(\alpha,\lambda)})+\sum_\alpha U_q(\binom{K_\alpha; 0}{l_\alpha}-\binom{\langle \alpha,\lambda\rangle}{l_\alpha}_{d_\alpha})\right),
\]
and we let $1_\lambda$ denote the corresponding cyclic generator.  For $a$ and $b$ in $U_q$ of respective $Q$-degrees $\mu$ and $\nu$, the multiplication on the modified algebra is given by
\[
(a\ \! 1_\lambda) (b\ \! 1_\tau)=a b\!\ 1_{\lambda-\nu}1_\tau=\delta_{\tau,\lambda-\nu}ab\ \! 1_\tau.
\]
We write $\dotU_q$ for the algebra $\dotU_q(G)$ when no confusion will arise.

\begin{remark}
Colloquially, the torus in $U_q(\mfk{g})$ is absorbed by the idempotent $1_\lambda$ in each $U_q1_\lambda$, and one is left only with the positive and negative subalgebras.  The modified algebra $\dotU_q$ can then be thought of as Lusztig's divided power algebra $U_q(\mfk{g})$, with the toral portion replaced by the algebra of idempotents $\oplus_{\lambda\in X} \mbb{C}1_\lambda$.  Note that the modified algebra is formally non-unital, as the unit element $\sum_{\lambda\in X}1_\lambda$ does not lie in $\dotU_q$.
\end{remark}

The algebra $\hatU_q$ is a pro-finite, linear topological Hopf algebra~\cite[\S 1.10]{egno15}, and we may identify $\hatU_q$ explicitly with the limit
\begin{equation}\label{eq:429}
\hatU_q={\varprojlim}_{cof}\dotU_q/I
\end{equation}
where $cof$ is the collection of cofinite ideals $I$ in $\dotU_q$, i.e.\ ideals for which the quotient $\dotU_q/I$ is finite-dimensional.
\par

For the moment, let us fix $\hatU_q$ to be the limit of \eqref{eq:429}, and denote the corresponding algebra of global operators by $\End_{\Fun/\mbb{C}}(F)$, were $F:\rep G_q\to Vect$ is the usual forgetful functor.  To understand the identification \eqref{eq:429}, note that any element $a$ in the completion $\hatU_q$ provides a natural endomorphism $a_?=a\cdot-\in \End_{\Fun/\mbb{C}}(F)$ given by left multiplication by $a$.  We therefore have a map of algebras $\hatU_q\to \End_{\Fun/\mbb{C}}(F)$, $a\mapsto a_?$, which one can check is an isomorphism, and so provides the claimed identification.
\par

Now, we have the global operators $E_\alpha$, $F_\alpha$, $E_\alpha^{(l_\alpha)}$, $F_\alpha^{(l_\alpha)}$, as well as the projection operators $1_\lambda$ for each $\lambda\in X$, and these operators topologically generate $\hatU_q$.  Furthermore, any (infinite) sum $\sum_{\lambda\in X}c_\lambda 1_\lambda$, $c_\lambda\in \mbb{C}$, provides a well-defined global operator on $\rep G_q$.  So, the product algebra $\prod_{\lambda\in X} \mbb{C}1_\lambda$, which is identified with the collection of arbitrary $\mbb{C}$-valued functions $\Fun(X,\mbb{C})$ on $X$, is naturally realized as a subalgebra in the algebra $\hatU_q$.

\begin{remark}
The completion $\hatU_q$ is the linear dual of the finite dual $(\dotU_q)^\circ$ \cite[Definition 1.2.3]{montgomery93}, which has $\rep G_q=\corep(\dotU_q)^\circ$.
\end{remark}

\subsection{Coherence of function algebras on groups}

Recall that an algebra $A$ is called {\it coherent} if the category of finitely presented $A$-modules is an abelian subcategory in the category of arbitrary $A$-modules.  We would like to work with general affine group schemes at some points, and so include the following result.

\begin{lemma}\label{lem:coherence}
The algebra of global functions $\O(\Pi)$ on any affine group scheme $\Pi$ is coherent.
\end{lemma}

\begin{proof}
Since $\O=\O(\Pi)$ is locally finite, as a coalgebra, we have that $\O$ is the direct limit (union) of its finitely generated, and hence Noetherian, Hopf subalgebras $\O=\varinjlim_\alpha \O_\alpha$.  Since extensions of commutative Hopf algebras are (faithfully) flat~\cite[Theorem 5]{takeuchi79}, $\O_\beta$ is flat over $\O_\alpha$ when $\alpha\leq \beta$.  It follows that $\O=\varinjlim_\alpha \O_\alpha$ is coherent~\cite[Theorem 2.3.3]{glaz06}.
\end{proof}

\section{Additional structures on the character lattice}

We introduce some basic structures on the character lattice, of a given almost simple group, which are employed throughout this work.  Below we consider an almost simple algebraic group $G$ with character lattice $X$, root lattice $Q$, and weight lattice $P$.

\subsection{(Strongly) admissible lattices}
\label{sect:admissible}

Given an intermediate lattice $Q\subset X\subset P$ between the root lattice and weight lattice in a given Dynkin type, and $q$ a $2l$-th root of $1$, we define
\[
X^{\M}:=\{x\in X:(x,y)\in l\mbb{Z}\ \forall\ y\in X\}.
\]
This is a sublattice in $X$.  Note that the restriction $\Omega|_{X^{\M}\times X^{\M}}$ takes values $\{\pm 1\}$.

\begin{definition}
We say the lattice $X$ is admissible at $q$ if $\Omega(x,x)=1$ for all $x\in X^{\M}$.  We call $X$ strongly admissible at $q$ if the restriction $\Omega|_{X^{\M}\times X^{\M}}$ is of constant value $1$.
\end{definition}

This is a technical condition which, it turns out, determines the nature of the M\"uger center of the quantum group $\rep G_q$.  In particular, the character lattice for $G$ is admissible if and only if the M\"uger center of $\rep G_q$ is Tannakian, and strongly admissible if and only if the braiding on the M\"uger center in $\rep G_q$ is the trivial vector space symmetry.  Rather, the lattice is strongly admissible if and only if the given fiber functor $\rep G_q\to Vect$, which is not itself a braided tensor functor, restricts to a \emph{symmetric} fiber functor on the M\"uger center of $\rep G_q$.

\begin{lemma}\label{lem:417}
Fix a Dynkin type with corresponding root and weight lattices $Q$ and $P$ respectively.  The following hold:
\begin{enumerate}
\item The simply-connected lattice $X_{sc}=P$ is admissible at arbitrary (even order) $q$ in all Dynkin types.
\item The simply-connected lattice in types $A_1$, i.e.\ the lattice for $\SL_2$, is strongly admissible at arbitrary (even order) $q$.
\item In types $A_{>1}, B, D, E$, and $G_2$, the simply-connected lattice $X_{sc}$ is strongly admissible if and only if $4\mid \ord(q)$.
\item In type $C_{>2}$, $X_{sc}$ is strongly admissible if and only if $4\nmid\ord(q)$, i.e.\ $2$ appears with multiplicity one in the prime decomposition of $\ord(q)$, or $8\mid \ord(q)$.
\item In type $F_4$, $X_{sc}$ is strongly admissible if and only if $8\mid \ord(q)$.
\item When $2\exp(P/Q)\mid l$ and $q$ is of order $2l$, all intermediate lattices $Q\subset X\subset P$ are admissible.
\item The lattice for $\PSL_2$ is strongly admissible when $4\nmid\ord(q)$ or $8\mid \ord(q)$, and inadmissible otherwise.
\end{enumerate}
\end{lemma}

\begin{proof}
Take $2l=\ord(q)$.  (1) In this case $X^{\M}=\mbb{Z}\{l_\alpha\alpha:\alpha\in \Delta\}$, and we calculate for an arbitrary element
\[
\begin{array}{rl}
(\sum_i c_il_i\alpha_i,l\sum_i c_il_i\alpha_i)& =l_i^2c_i^2(\alpha_i,\alpha_i)+2l_il_j\sum_{i<j}c_ic_j(\alpha_i,\alpha_j)\\
&=2ll_ic_i^2+2ll_j\sum_{i<j}c_ic_j\langle\alpha_i,\alpha_j\rangle\in 2l\mbb{Z}.
\end{array}
\]
Whence we have admissibility.  (2) Here we have $X^{\M}=lQ=l\mbb{Z}\alpha$, and the computation $(l\alpha,l\alpha)=2l^2$ implies strong admissibility for $\SL_2$.
\par

(3) In the simply-laced case we have $X^{\M}=lQ$ and $(la,lb)\in l^2(a,b)$ for $a,b\in Q$.  When $2\mid l$ this implies strong admissibility.  When $2\nmid l$ if we take neighbors then $(l\alpha,l\beta)=l^2\notin 2l\mbb{Z}$, obstructing strong admissibility.  In type $B_n$ we find a similar obstruction to strong admissibility when $2$ does not divide $l$.  When $2\mid l$ and $\beta$ is short we have again $(l_\alpha\alpha,l\beta)=l^2\langle \alpha,\beta\rangle\in 2l\mbb{Z}$, and for the unique long $\gamma$,
\[
(l_\gamma\gamma,l_\gamma\gamma)=ll_\gamma\langle \gamma,\gamma\rangle =2ll_\gamma\in 2l\mbb{Z}.
\]
So $(X^{\M},X^{\M})\subset 2l\mbb{Z}$ and we have strong admissibility.  For $G_2$, with short root $\alpha$ and long root $\gamma$, 
\[
(l\alpha,l\alpha)=2l^2,\ (l_\gamma\gamma,l\alpha)=l^2\ \text{or}\ 3l^2,\ (l_\gamma\gamma,l_\gamma\gamma)=l^2\text{ or }3l^2,
\]
depending on if $3\mid l$ or not, implying failure of strong admissibility when $l$ is odd and establishing strong admissibility when $l$ is even.
\par

(4) The Killing form on $Q$ takes values in $2\mbb{Z}$ in type $C_n$.  When $l$ is odd $l_\alpha=l$ for all simple $\alpha$, and $X^{\M}=lQ$.  So $(X^{\M},X^{\M})=l^2(Q,Q)\in 2l\mbb{Z}$ in this case, and we have strong admissibility.  When $l$ is even $l_\alpha=l/2$ for all long roots and $l_\beta=l$ for the unique short root $\beta$.  When $4\mid l$ this is sufficient to establish strong admissibility, and in the remaining case when $2$ appears with multiplicity $1$ in $l$ we can take neighboring long roots $\alpha$ and $\beta$ to find $(l_{\alpha}\alpha,l_{\beta}\beta)=l_\beta l\notin 2l\mbb{Z}$.  (5) One basically combines the arguments for types $B$ and $C$ to observe the claim for $F_4$, as we have both short neighbors and long neighbors.  We leave (6) and (7) to the interested reader, as they are just illustrative examples.
\end{proof}

\subsection{Balancing functions}
\label{sect:balancing}

\begin{definition}
A balancing function on the character lattice $X$ for $G$, at a given parameter $q$, is a function $\omega:X\times X\to \mbb{C}^\times$ with the following properties:
\begin{enumerate}
\item[(a)] $\omega$ is $X$-linear in the first coordinate.
\item[(b)] In the second coordinate, $\omega$ satisfies the $X^{\M}$-semilinearity $\omega(a,a'+x)=q^{-(a,x)}\omega(a,a')$, for $x\in X^{\M}$.
\item[(c)] The restriction to $X^{\M}\times X$ is trivial, $\omega|_{X^{\M}\times X}\equiv 1$.
\end{enumerate}
\end{definition}

Note that we may view $\omega$ as a map from the quotient $(X/X^{\M})\times X$ satisfying the prescribed (semi)linearities.  Also, by strong admissibility, the function $q^{-(-,x)}$ is trivial on $X^{\M}$, so that the conditions (b) and (c) are not in conflict.

\begin{lemma}\label{lem:omega}
Every strongly admissible character lattice admits a balancing function.
\end{lemma}

\begin{proof}
Consider any set theoretic section $s:Z=(X/X^{\M})\to X$.  Then each element $a\in X$ admits a unique expression $a=x+sz$ with $x\in X^{\M}$ and $z\in Z$, and we may define the desired function $\omega$ by $\omega(a,a')=\omega(a,x+sz):=q^{-(a,x)}$.
\end{proof}

\section{The log-modular kernel as a quasi-Hopf algebra}
\label{sect:uM}

We provide explicit presentations of the quasi-Hopf kernels $u^{\M}_q(G)$, for almost simple $G$ with strongly admissible character lattice $X$.  We first introduce $u^{\M}_q(G)$ as an associative algebra, then provide its quasi-Hopf structure, $R$-matrix, and ribbon element when applicable.  We leave a proof of factorizability to Section~\ref{sect:fiber}.  As we will see, the quasi-Hopf structure on $u^{\M}_q(G)$ is not unique, but depends on a choice of balancing function on the character lattice for $G$.
\par

We note that the materials of this section are relatively independent of the materials of the sections that follow.  What we give here is a direct, algebraic, construction of the log-modular kernel.  In the remainder of the paper we provide both categorical and representations theoretic (re)productions of this same object, and investigate some consequences of these varying perspectives in Section~\ref{sect:(1,p)}.

\subsection{The log-modular kernel as an associative algebra~\cite{arkhipovgaitsgory03}}

Consider again the linear topological Hopf algebra $\hatU_q(G)={\varprojlim}_{cof} \dotU_q(G)/I$ of global operators for $\rep G_q$, as in Section~\ref{sect:global_op}.  We let $Z$ denote the quotient $Z=X/X^{\M}$.  As explained in Section~\ref{sect:global_op}, arbitrary $\mbb{C}$-valued functions on $X$ determine global operators on $\rep G$, so that characters $\chi$ on $Z$ in particular are identified with operators $\sum_{\lambda\in X}\chi(\lambda)1_\lambda\in \hatU_q$.  We employ the distinguished grouplikes $K_\alpha\in \Fun(X,\mbb{C})\subset \hatU_q$ below, which are precisely the functions $K_\alpha:X\to \mbb{C}^\ast$, $\lambda\mapsto q^{(\alpha,\lambda)}$.

\begin{definition}
Define $u^{\M}_q(G)$, as an associative algebra, to be the subalgebra in $\hatU_q(G)$ generated by the operators $\mathsf{E}_\alpha:=K_\alpha E_\alpha$ and $F_\alpha$, for $\alpha$ simple, as well as the functions $\mbb{C}[Z^\vee]$ on the quotient $Z$.
\end{definition}

One has relations between the characters $Z^\vee$ and the generators $\mathsf{E}_\alpha$, $F_\alpha$ as follows: for $\chi\in Z^\vee$ and any simple $\alpha$ we have
\[
\chi\mathsf{E}_\alpha\chi^{-1}=\chi(\alpha)\mathsf{E}_\alpha,\ \ \chi F_\alpha\chi^{-1}=\chi^{-1}(\alpha)F_\alpha,
\]
and we have those relations between the $\mathsf{E}$'s and $F$'s which are implied by the usual quantum Serre relations \cite[\S 1]{lusztig90II}.  One can check that the Serre relations for the usual positive elements $E_\alpha\in \dotU_q$ imply the exact same (Serre) relations for $\mathsf{E}_\alpha$.  Only the commutator relations for $\mathsf{E}_\alpha$ and $F_\beta$ are altered, due to the presence of $K_\alpha$ in the formula $\mathsf{E}_\alpha=K_\alpha E_\alpha$.  We claim, and prove in Lemma~\ref{lem:triangle} below, that these relations provide a complete list of relations for $u^{\M}_q(G)$.

\begin{remark}
Note that the distinguished grouplikes $K_\alpha$ do not lie in $Z^\vee\subset \Fun(X,\mbb{C})$ in general.  For example, for $\SL(N)$ at $N>2$ we have $X^M=lQ$ and $q^{(\alpha,l\beta)}=q^{-l}=-1$ whenever $\alpha$ and $\beta$ are neighbors, so that $K_\alpha$ does not satisfy the required vanishing on $X^M$.  However, the squares $K_\alpha^2$ always lie in $Z^\vee$.
\end{remark}

\begin{remark}
The algebra $u^M_q(G)$ is the same as the algebra of~\cite{arkhipovgaitsgory03}, given there as the algebra of coinvariants in $\hatU_q$ with respect to the quantum Frobenius (see Section~\ref{sect:qfrob1}), and $\rep u^{\M}_q(G)$ is the category $_k\msc{C}_{G_1}$ of~\cite[\S 3.11]{andersenparadowski95}.
\end{remark}

Let $\dotu_q$ denote the subalgebra in $\dotU_q$ generated by the idempotents $1_\lambda$ and the elements $E_\alpha 1_\lambda$, $F_\alpha 1_\lambda$, for arbitrary $\lambda\in X$ and simple $\alpha$.  This is the modified small quantum group, and its representations $\rep \dotu_q$ are $X$-graded vector spaces with operators $E_\alpha$ and $F_\alpha$, $\alpha\in \Delta$, which satisfy the quantum Serre relations.
\par

We may consider the cofinite completion $\hatu_q$, i.e.\ the algebra of endomorphisms of the fiber functor for $\rep\dotu_q$.  By considering the ideals $I_N$ in $\dotu_q$ generated by the idempotents $\{1_\lambda:|\lambda|\geq N\}$, $N\geq 0$, one can calculate directly that the completed algebra is simply the product
\[
\hatu_q=\varprojlim_N \dotu_q/I_N=\prod_{\lambda\in X} u_q1_\lambda.
\]
Here the $u_q1_\lambda$ are defined as in Section~\ref{sect:global_op}, with $u_q$ the subalgebra of $U_q$ generated by the $E_\alpha$, $F_\alpha$, and all toral elements.

\begin{lemma}\label{lem:365}
The restriction functors $\rep G_q\to \rep\dotu_q$ is surjective (in the sense of~\cite{egno15}).
\end{lemma}

We employ in the proof a certain basic understandings of dominant weights, and the lattice $X^M$, from Section~\ref{sect:qfrob}.  We have elected to reference the necessary results from Section~\ref{sect:qfrob} when appropriate, rather than delay the proof.

\begin{proof}
The surjective image of $\rep G_q$ in $\rep\dotu_q$ is the smallest subcategory in $\rep\dotu_q$ which contains the image of $\rep G_q$ and is closed under taking subobjects and quotients.  This subcategory is closed under duality in $\rep\dotu_q$ and, since the tensor product on $\rep\dotu_q$ is biexact, it is also closed under taking tensor products.  That is to say, the surjective image is an embedded tensor subcategory in $\rep\dotu_q$.  We have proposed that the surjective image of $\rep G_q$ is all of $\rep\dotu_q$.
\par

We let $L(\lambda)$ denote the simple in $\rep G_q$ of highest (dominant) weight $\lambda\in X^+$~\cite[Proposition 6.4]{lusztig89}.  We have the Steinberg module $St=L((l-1)\rho)$, which is simple, self-dual, and projective in $\rep G_q$.  The image of $St$ in $\rep \dotu_q$ remains projective~\cite[Theorem 4.3]{andersenpolowen92}.  We claim now that all simples in $\rep\dotu_q$ appear as subquotients of simples in $\rep G_q$.  Simples in $\rep\dotu_q$ are determined by their highest weights $\msc{L}(\nu)$, which are now associated to arbitrary elements $\nu\in X$, and so we see that $\msc{L}(\lambda)$ is a quotient $L(\lambda)$ for any dominant $\lambda$.  When $\mu\in X^M$, the simple $\msc{L}(\mu)$ is $1$-dimensional.
\par

The lattice $X^{\M}$ is itself the character lattice of a certain dual group to $G$, and $(X^{\M})^+=X^{\M}\cap X^+$ (see Section~\ref{sect:qfrob1}).  Since, $X^{\M}$ is generated by its dominat weights (see Proposition~\ref{prop:435}), we find that $\msc{L}(\mu)$ is in the surjective image of $\rep G_q$ whenever $\mu\in X^{\M}$.  Since $X^{\M}$ contains some positive multiple of all the fundamental weights we have that all $\lambda\in X$ are in the $X^{\M}$-orbit of the dominant weights $X^+$.  Rather, $X=X^{\M}+X^+$, and since each $1$-dimensional $\msc{L}(\mu)$ is a tensor unit we obtain
\[
\operatorname{Irrep}(\dotu_q)=\{\msc{L}(\nu):\nu\in X\}=\{\msc{L}(\mu)\ot \msc{L}(\lambda):\mu\in X^{\M},\ \lambda\in X^+\}.
\]
So all of the simples are in the surjective image of $\rep G_q$ in $\rep \dotu_q$.  By tensoring with the projective $St$, we find further that the surjective image contains a projective $\msc{P}(\nu)$ which surjects onto each simple $\msc{L}(\nu)$.  By considering composition series, it follows that each object $V$ in $\rep\dotu_q$ admits a surjection $\msc{P}\to V$ from a projective in the surjective image of $\rep G_q$.  Hence the surjective image is all of $\rep\dotu_q$.
\end{proof}

Lemma~\ref{lem:365} says, equivalently, that the completion $\hatu_q\to \hatU_q$ of the inclusion $\dotu_q\to \dotU_q$ is injective~\cite[Lemma 2.2.13]{schauenburg92}.  Since the subalgebra $u^{\M}_q\subset \hatU_q$ lies in $\hatu_q$, we may replace $\hatU_q$ with $\hatu_q$ in our analysis of the linear structure of $u^{\M}_q$.
\par

In the following Lemma we consider $u^+_q(G)$ as the subalgebra of $\hatu_q$ generated by the $\mathsf{E}_\alpha$, and let $u^-_q(G)$ denote the subalgebra generated by the $F_\alpha$.

\begin{lemma}\label{lem:triangle}
The subalgebra $u^+_q(G)$ (resp.\ $u^-_q(G)$) in $u^{\M}_q(G)$ has the expected presentation, with generators $\mathsf{E}_\alpha$ (resp.\ $F_\alpha$) and the quantum Serre relations of~\cite{lusztig90II}.  Furthermore, multiplication provides a triangular decomposition
\begin{equation}\label{eq:588}
u^-_q(G)\ot\mbb{C}[Z^\vee]\ot u^+_q(G)\overset{\cong}\to u^{\M}_q(G).
\end{equation}
\end{lemma}

\begin{proof}
The Serre relations for $u^+_q(G)$ imply that $u^+_q$ has a spanning set in terms of ordered monomials in the root vectors $\mathsf{E}_\gamma$~\cite{lusztig90II}.  The algebra $u^+_q$ has precisely these relations if and only if the root vector monomials provide a basis for this algebra.  However, this follows by the (topological) basis of $\hatu_q$ in terms of monomials in the root vectors~\cite[\S 31.1.2, 36.2.1]{lusztig93}.  A similar argument establishes the desired result for $u^-_q$.
\par

As for the triangular decomposition, the commutator relations between the $\mathsf{E}_\alpha$ and $F_\beta$ imply that the map~\eqref{eq:588} is surjective, and injectivity follows again by the basis of $\hatu_q$ in terms of monomials in root vectors.
\end{proof}

\subsection{The quasi-Hopf structure on $u^{\M}_q(G)$ via a balancing function}

We introduce a (family of) quasi-Hopf structure(s) on $u^{\M}_q(G)$, determined by a choice of balancing function for the character lattice $X$.  We refer the unfamiliar reader to~\cite{majid98} for details on quasi-Hopf algebras, or any other standard reference.
\par

Fix a balancing function $\omega$, with pointwise inverse $\omega^{-1}$.  We have $\omega(1,\ast)=\omega(\ast,1)=1$, and hence $\omega$ defines a (non-Drinfeld) twist
\[
\omega=\sum_{\lambda,\mu\in X}\omega(\lambda,\mu)1_\lambda\ot 1_\mu\in \Fun(X,\mbb{C})\hat{\ot}\Fun(X,\mbb{C})\subset \hatU_q\hat{\ot}\hatU_q.
\]
Whence we may twist in the usual fashion to obtain a new quasi-Hopf algebra $\hatU_q^\omega$ with the same (linear topological) algebra structure, comultiplication
\[
\nabla:=\omega^{-1}\Delta(-)\omega
\]
and associator
\[
\phi:=(1\ot\omega)^{-1}(1\ot\Delta)(\omega)^{-1}(\Delta\ot 1)(\omega)(\omega\ot 1).
\]
We have also the normalized antipode $(S^\omega,1,\beta)$, where
\[
S^\omega(x)=\tau^{-1}S(x)\tau,\ \ \beta=(\sum_{\lambda\in X}\omega^{-1}(\lambda,-\lambda)1_\lambda)\tau=\sum_\lambda\omega^{-1}(\lambda,-\lambda)\omega^{-1}(\lambda,\lambda)1_\lambda,
\]
and $\tau=\sum_{\lambda\in X}\omega(-\lambda,\lambda)1_\lambda=\sum_\lambda\omega^{-1}(\lambda,\lambda)1_\lambda$.  We will establish the following.

\begin{proposition}\label{prop:437}
The subalgebra $u^{\M}_q(G)$ is a quasi-Hopf subalgebra in $\hatU_q^\omega$, for any choice of $\omega$.  The formula for the comultiplication $\nabla$ on $u^{\M}_q(G)$ is as described in Lemma~\ref{lem:468} below.
\end{proposition}

We choose a section $s:Z\to X$ and identify $Z$ with its image in $X$ in the formulas below.  We can understand $\phi$ and $\beta$ as functions from $X^3$ and $X$ respectively.  We have
\[
\phi:X^3\to \mbb{C},\ \ \begin{array}{rl}
\phi(a,b,c)&=\omega^{-1}(b,c)\omega(a+b,c)\omega^{-1}(a,b+c)\omega(a,b)\\
&=\omega(a,c)\omega^{-1}(a,b+c)\omega(a,b).
\end{array}
\]
By linearity of $\omega$ in the first component, and $X^{\M}$-semilinearity in the second component we see that
\[
\phi(a+x,b,c)=\phi(a,b+x,c)=\phi(a,b,c+x)=\phi(a,b,c)\ \ \text{for}\ x\in X^{\M}.
\]
So $\phi$ is constant on $X^{\M}$-cosets in each component, and thus is identified with a function from the quotient $Z^3$,
\[
\phi:Z^3\to \mbb{C}, \ \ \phi(z,z',z'')=\omega(z,z'')\omega^{-1}(z,z'+z'')\omega(z,z').
\]
One also observes directly that $\beta$ is constant on $X^{\M}$-cosets to find that it is identified with a function on $Z$, $\beta(z)=\omega^{-1}(z,-z)\omega^{-1}(z,z)$.  This information implies the following.

\begin{lemma}\label{lem:461}
Let $1_z\in \mbb{C}[Z^\vee]$ denote the idempotent associated to an element $z\in Z$.  We have $\phi\in \mbb{C}[Z^\vee]^{\ot 3}\subset u^{\M}_q(G)^{\ot 3}$ and $\beta\in \mbb{C}[Z^\vee]$.  Specifically,
\[
\phi=\sum_{z\in Z}\omega(z,z'')\omega^{-1}(z,z'+z'')\omega(z,z')1_z\ot 1_{z'}\ot 1_{z''},\ \ \beta=\sum_{z\in Z}\omega^{-1}(z,z)\omega^{-1}(z,-z)1_z.
\]
\end{lemma}

Let us define for $\gamma\in X$ functions $\mathcal{L}_\gamma,L_\gamma:X\to \mbb{C}$ by
\[
\mathcal{L}_\gamma(\lambda):=q^{-(\gamma,\lambda)}\omega(\gamma,\lambda),\ \ L_\gamma(\lambda):=\omega(\lambda,\gamma).
\]
These functions are constant on $X^{\M}$-cosets and hence provide elements in $\mbb{C}[Z^\vee]\subset u^{\M}_q$.  We define also the interior product
\[
\iota_\gamma\phi:X^2\to \mbb{C},\ \ \iota_\gamma\phi(\lambda,\mu):=\phi(\lambda,\mu,\gamma).
\]
This function is also constant on $X^{\M}$-cosets so that $\iota_\gamma\phi\in\mbb{C}[Z^\vee]^{\ot 2}$.

\begin{lemma}\label{lem:468}
In $\hatU_q^\omega$ we have $\nabla(\xi)=\xi\ot \xi$ for all $\xi\in Z^\vee$,
\[
\nabla(\mathsf{E}_\alpha)=\mathsf{E}_\alpha\ot \mathcal{L}^{-1}_\alpha+\iota_{-\alpha}\phi^{-1} L_{-\alpha} K_\alpha^2\ot \mathsf{E}_\alpha,
\]
\[
\text{and}\ \ \nabla(F_\alpha)=F_\alpha\ot \mathcal{L}_\alpha+\iota_{\alpha}\phi^{-1} L_\alpha\ot F_\alpha.
\]
Furthermore, $u^{\M}_q(G)$ is stable under the application of the antipode $S^\omega$.
\end{lemma}

\begin{proof}
The equality $\nabla(\xi)=\xi\ot \xi$ follows from the fact that $\omega$ commutes with elements in $Z^\vee$.  Now, once calculates directly
\[
\begin{array}{l}
\nabla(\mathsf{E}_\alpha)=\omega^{-1}\Delta(\mathsf{E}_\alpha)\omega\\
=\sum_{\lambda,\mu\in X} \omega^{-1}(\lambda+\alpha,\mu)\omega(\lambda,\mu)\mathsf{E}_\alpha 1_\lambda\ot K_\alpha 1_\mu+\omega^{-1}(\lambda,\mu)\omega(\lambda,\mu-\alpha) 1_\lambda K_\alpha^2\ot 1_\mu \mathsf{E}_\alpha\\
=\sum_{\lambda,\mu}q^{(\alpha,\mu)}\omega^{-1}(\alpha,\mu)\mathsf{E}_\alpha 1_\lambda\ot 1_\mu+\phi^{-1}(\lambda,\mu,-\alpha)L(\lambda) 1_\lambda K_\alpha^2\ot 1_\mu\mathsf{E}_\alpha\\
=\mathsf{E}_\alpha\ot \mathcal{L}^{-1}_\alpha+\iota_{-\alpha}\phi^{-1} L_{-\alpha} K_\alpha^2\ot \mathsf{E}_\alpha.
\end{array}
\]
Similarly,
\[
\begin{array}{l}
\nabla(F_\alpha)=\omega^{-1}\Delta(F_\alpha)\omega\\
=\sum_{\lambda,\mu}q^{-(\alpha,\mu)}\omega^{-1}(\lambda-\alpha,\mu)\omega(\lambda,\mu)F_\alpha 1_\lambda\ot 1_\mu+\omega^{-1}(\lambda,\mu)\omega(\lambda,\mu+\alpha) 1_\lambda\ot 1_\mu F_\alpha\\
=\sum_{\lambda,\mu}q^{-(\alpha,\mu)}\omega(\alpha,\mu)F_\alpha 1_\lambda\ot 1_\mu+\phi^{-1}(\lambda,\mu,\alpha)L_\alpha(\lambda) 1_\lambda\ot 1_\mu F_\alpha\\
=F_\alpha\ot \mathcal{L}_\alpha+\iota_{\alpha}\phi^{-1} L_\alpha\ot F_\alpha.
\end{array}
\]
For the antipose we have $S^\omega(\xi)=\xi$,
\[
\begin{array}{l}
S^\omega(\mathsf{E}_\alpha)=-\left(\sum_{\lambda\in X} q^{-(\lambda,\alpha)}\omega(\lambda,\lambda)\omega(\lambda-\alpha,\lambda-\alpha)\right)K_\alpha^{-2}\mathsf{E}_\alpha,\\\\
S^\omega(F_\alpha)=-\left(\sum_{\lambda\in X}q^{(\lambda,\alpha)}\omega(\lambda,\lambda)\omega^{-1}(\lambda+\alpha,\lambda+\alpha)\right) F_\alpha.
\end{array}
\]
One can check directly that these coefficients are constant on $X^{\M}$-cosets in $X$, and hence lie in $\mbb{C}[Z^\vee]$.
\end{proof}

We can now prove the proposition.

\begin{proof}[Proof of Proposition~\ref{prop:437}]
Follows from Lemmas~\ref{lem:461} and~\ref{lem:468}.
\end{proof}

\subsection{The ribbon structure on $u^{\M}_q(G)$}

Fix $\omega$ a balancing function, as above.  We have the standard $R$-matrix $R^\omega=\omega_{21}^{-1}R\omega$ for the twisted algebra $\hatU_q^\omega$.  The following lemma is verified by straightforward computation.

\begin{lemma}
The $R$-matrix $R^\omega$ lies in $u^{\M}_q(G)\ot u^{\M}_q(G)$, and hence provides $u^{\M}_q(G)$ with a quasitriangular structure.
\end{lemma}

By categorical considerations~\cite[\S 8.9]{egno15}, the Drinfeld elements for $\hatU^\omega_q$, and hence $u^{\M}_q$, is given by the formula $\tau^{-1}S(\tau)u$, where $u$ is the Drinfeld element for $\hatU_q$.  The pivotal structure on $\hatU_q$, which is given by multiplication by the grouplike $K_\rho$ where $\rho=\sum_{\gamma\in \Phi^+}\gamma$, provides a pivotal structure for the twist $\hatU_q^\omega$, which is given by multiplication by $\tau^{-1}S(\tau)K_\rho$.  Hence the ribbon element for $\hatU_q^\omega$ is
\[
v^\omega=\tau^{-1}S(\tau)K_\rho (\tau^{-1}S(\tau)u)^{-1}=K_\rho u^{-1}=v,
\]
where $v$ is untwisted ribbon element for the quantum group.  (We use the fact that $\tau$ is in $\Fun(X,\mbb{C})$ and hence commutes with $u$.)
\par

When $X$ is the simply-connected lattice, so that $X^{\M}=lQ$, it is easy to see that $K_\rho\in Z^\vee$.  More generally, $K_\rho$ is in $Z^\vee$ whenever $K_\rho|_{X^{\M}}\equiv 1$.  Since $\tau$ is a function on $X$, $S(\tau)=\tau^{-1}$ and $\tau^{-1} S(\tau)=\tau^{-2}$.  This element $\tau^{-2}$ is constant on $X^{\M}$-cosets and hence in $\mbb{C}[Z^\vee]$.  Thus the pivotalizing element $\tau^{-1}S(\tau)K_\rho$ for $\hatU_q^\omega$ lies in $u^{\M}_q$ whenever $K_\rho|_{X^{\M}}\equiv 1$.

\begin{proposition}\label{prop:682}
Suppose that $X$ is the simply-connected lattice, or that $K_\rho|_{X^{\M}}\equiv 1$.  Then for any choice of balancing function, the induced quasi-Hopf structure on $u^{\M}_q(G)$ naturally extends to a ribbon structure under which the ribbon element $v$ is just the standard ribbon element for the large quantum group $\hatU_q$.
\end{proposition}

If one considers the example $(\PSL_2)_q$, we see that $K_\rho|_{X^{\M}}\equiv 1$ when $l$ is odd, since $X^{\M}=lQ$ in this case, and $K_\rho|_{X^{\M}}$ is not identically $1$ when $4\mid l$, as $X^{\M}=\frac{l}{2}Q$ and $K_\rho(\frac{l}{2}\alpha)=-1$.  So the induced ribbon structure on $u^{\M}_q(G)$ is not exclusive to the simply-connected case, but fails to hold in general.  We continue our discussion of quantum $\PSL_2$ in Section~\ref{sect:remPSL2}.
\par

Of course, as a quasi-Hopf algebra, the definition of $u^{\M}_q(G)$ depends on a choice of balancing function $\omega$.  However, by Proposition~\ref{prop:qhuM_exists} below, the braided tensor category $\rep u^{\M}_q(G)$ is independent of choice of balancing function, up to braided equivalence and ribbon equivalence when applicable.  We find in Corollary~\ref{cor:nondegen} that $u^{\M}_q(G)$, with $R$-matrix as above, is in fact factorizable, and hence log-modular.

\subsection{The log-modular kernel for $\mfk{sl}_2$}
\label{sect:sl2}

Consider $u^{\M}_q(\mfk{sl}_2):=u^{\M}_q(\SL_2)$.  The character $K=K_\alpha:X_{sc}\to \mbb{C}$, $K(\lambda)=q^{(\lambda,\alpha)}$, is of constant value $1$ on $X^{\M}=l\mbb{Z}\alpha$.  Hence $K\in u^{\M}_q(\mfk{sl}_2)$, and therefore $E=K^{-1}\mathsf{E}$ is in $u^{\M}_q(\mfk{sl}_2)$.  Therefore
\[
u^{\M}_q(\mfk{sl}_2)=\left\{\begin{array}{c}
\text{the standard subalgebra in $U_q(\mfk{sl}_2)$ generated by}\\
\text{the $E$, $F$, and $K$, as an associative algebra}
\end{array}\right\}.
\]
So we see that $u^{\M}_q(\mfk{sl}_2)$ simply consists of a new choice of comultiplication, associator, and ribbon structure, on the usual small quantum group in $U_q(\mfk{sl}_2)$.

\section{Quantum Frobenius and the M\"uger center of $\rep G_q$}
\label{sect:qfrob}

We now turn our attention from the quasi-Hopf algebra $u^{\M}_q(G)$ to the canonical form $(\rep G_q)_{G^\vee}$ highlighted in the introduction.  In this section and all following section, $q$ is a root of unity of even order $2l$ and $G$ is an almost simple algebraic group with strongly admissible character lattice $X$ at $q$.

\subsection{The quantum Frobenius}
\label{sect:qfrob1}

Define the dual group $G^\vee$ to $G$ at $q$ to be the almost simple algebraic group with the following Cartan data:
\begin{itemize}
\item The character lattice for $G^\vee$ is $X^{\M}$.
\item The simple roots for $G^\vee$ are $\Delta^\vee:=\{l_i\alpha_i:\alpha_i\in \Delta\}$
\item The Cartan integers are given by $a_{ij}^\vee=a_{ij}\frac{l_i}{l_j}$.
\end{itemize}
When all $d_i$ divide $l$ the group $G^\vee$ is of Langlands dual type to $G$, and $G^\vee$ is exactly the Langlands dual when $G$ is additionally simply-connected.  When the $d_i$ do not divide $l$ the dual group $G^\vee$ is of the same Dynkin type as $G$.
\par

For the algebra $\dotU_q=\dotU_q(G)=\bigoplus_{\lambda\in X}U_q 1_\lambda$ of~\cite[Chapter 23 \& 31]{lusztig93}, which has $\rep\dotU_q=\rep G_q$, we have the quantum Frobenius map
\[
\operatorname{Fr}^\ast:\dotU_q(G)\to \dotU(G^\vee),\ \ \left\{\begin{array}{l}
E_\alpha\mapsto 0\\
F_\alpha\mapsto 0\\
E_\alpha^{(l_\alpha)}\mapsto e_\alpha\\
F_\alpha^{(l_\alpha)}\mapsto f_\alpha\\
1_\lambda\mapsto 1_\lambda\ \ \text{if }\lambda\in X^{\M},\ \ 0\ \text{else},
\end{array}\right.
\]
which is a surjective map of quasi-triangular Hopf algebras~\cite[Theorem 35.1.9]{lusztig93}.  We note that $\dotU^\vee=\dotU(G^\vee)$ recovers classical representations for the dual group $\rep \dotU^\vee=\rep G^\vee$.

\begin{remark}
For $\SL_2$ and $\Sp_{2n}$ at odd $l$ the quantum Frobenius actually lands in the quasi-classical algebra $\dotU_{-1}^\vee$.  However, one can rescale the generators to obtain an identification $\dotU^\vee_{-1}=\dotU^\vee$ in these particular cases.  The important point in the strongly admissible setting is the identical vanishing of the $R$-matrix for $\rep \dotU^\vee_{\pm 1}$ which implies that the forgetful functor $\rep \dotU^\vee_{\pm 1}\to Vect$ is symmetric, and hence $\rep\dotU^\vee_{\pm 1}$ is directly identified with representations of an algebraic group via Tannakian reconstruction~\cite{deligne07,milne17}.
\end{remark}

Restricting along the quantum Frobenius Hopf map yields a braided tensor embedding
\[
\Fr:\rep G^\vee\to \rep G_q,
\]
which we also call the quantum Frobenius.  There is a third form of the quantum Frobenius, which is that of a Hopf inclusion to the quantum function algebra $\Fr_\ast:\O(G^\vee)\to \O_q(G)$, where $\O_q(G)=\operatorname{coend}(\rep G_q\to Vect)=\Hom_{Cont}(\hatU_q,\mbb{C})$.  One then recovers the categorical Frobenius by corestriction $\corep \O(G^\vee)\to \corep \O_q(G)$.
\par

To ease notation we generally write $\O$ for $\O(G^\vee)$ and $\O_q$ for $\O_q(G)$.

\begin{remark}
The algebra $\O_q$ is presumably the quantum function algebra of~\cite{lusztig90II,lusztig09}.
\end{remark}

\subsection{The quantum Frobenius and the M\"uger center of $\rep G_q$}

We aim to prove the following result.

\begin{theorem}\label{thm:calculatingZ}
The quantum Frobenius $\Fr:\rep G^\vee\to \rep G_q$ is an equivalence onto the M\"uger center of $\rep G_q$.
\end{theorem}

In order to prove the theorem we recall some basic representation theoretic facts.  Recall that a weight $\lambda\in X$ is called dominant if $\langle \alpha,\lambda\rangle\geq 0$ for all $\alpha\in \Delta$.  Equivalently, we may employ the Killing form to find that $\lambda$ is dominant if and only if $(\alpha,\lambda)\geq 0$ for all $\alpha$.  We let $X^+$ denote the set of dominant weights in $X$.
\par

By a standard analysis, the simples in $\rep G_q$ are classified up to isomorphism by their highest weights.  Given a weight $\lambda\in X$ which appears as a highest weight for some object in $\rep G_q$, and hence as the highest weight of some simple, we let $L(\lambda)$ denote the corresponding simple.

\begin{proposition}\label{prop:435}
For any simple $L(\lambda)$ in $\rep G_q$, the corresponding weight $\lambda$ is dominant.  Furthermore, the map $\operatorname{Irrep}G_q\to X^+$, $L(\lambda)\mapsto \lambda$, is a bijection.
\end{proposition}

\begin{proof}
One proceeds exactly as in the proof of~\cite[Proposition 6.4]{lusztig89}. 
\end{proof}

The following lemma is, without doubt, well-known and classical.

\begin{lemma}\label{lem:401}
The dominant weights $X^+$ span $X$.
\end{lemma}

\begin{proof}
Enumerate the simple roots $\Delta=\{\alpha_1,\dots,\alpha_n\}$ and define $S_j$ to be the set of $x\in X$ with $(\alpha_i,x)=0$ for all $i<j$, and $(\alpha_j,x)>0$.  Elements of $S_i$ are exactly those elements which have an expression in terms of fundamental weights in which the coefficients of $f_i$ are $0$, for all $i<j$, and the coefficient of $f_j$ is positive.  Note that $S_j\neq \emptyset$, since $P/X$ is finite, and hence some power of each fundamental weight lies in $X$.
\par

For each $1\leq j\leq n$ take $x_j\in S_j$ with minimal pairing with $\alpha_j$, $(\alpha_j,x_j)=\min\{(\alpha_j,x):x\in S_j\}$.  By replacing $x_j$ with a sum $x_j+\sum_{k>j}c_kx_k$ we may assume additionally that each $x_j$ is dominant.  Now, for arbitrary $\lambda\in X$ with $(\alpha_i,\lambda)=0$ for all $i<j$, our minimality assumption on $x_j$ implies that there is some $c_j(\lambda)\in \mbb{Z}$ with $(\alpha_j,\lambda-c_j(\lambda)x_j)=0$.  Whence we see, by induction, that for any $\lambda\in X$ one can take a difference $\lambda-\sum_i c_i(\lambda)x_i$ so that $(\alpha_j,\lambda-\sum_ic_i(\lambda)x_i)=0$ for all $j$.  By non-degeneracy of the Killing form on the rationalization $X_\mbb{Q}$ we see $\lambda=\sum_ic_i(\lambda)x_i$.  Hence $\{x_1,\dots,x_n\}$ provides a dominant spanning set for $X$.
\end{proof}

We can now prove our theorem.

\begin{proof}[Proof of Theorem~\ref{thm:calculatingZ}]
The image of the quantum Frobenius $\Fr:\rep G^\vee\to \rep G_q$ is the subcategory tensor generated by the simples $L(\lambda)$ with $\lambda\in (X^{\M})^+$.  One sees this directly from the definition of the associated surjection $\dotU_q\to \dotU^\vee$ and the classification of simples for $\dotU^\vee$.
\par

We note that for any extension $W$ of objects $V$ and $V'$ in the image of $\rep G^\vee$, the $X$-grading on $W$ is necessarily a grading by $X^{\M}$.  That is to say, $W_\lambda=0$ for all $\lambda\notin X^{\M}$.  This implies that $E_i,F_i:W\to W$ are trivial operators.  (One needs to use strong admissibility of $X$ here when $l=2$ in types $B$ and $C$, and $l=3$ in type $G_2$.)  Hence the action of $\dotU_q$ on $W$ factors through the Frobenius $\dotU_q\to \dotU^\vee$.  Rather, $W$ is in the image of $\rep G^\vee$, and we see that the image of $\rep G^\vee$ is closed under extension.  We can describe this image simply as the collection of $V$ in $\rep G_q$ with $X$-grading induced by a $X^{\M}$-grading.
\par

Now, take $L(\lambda)$ a simple in the M\"uger center of $\rep G_q$, and let $v_\lambda$ be a highest weight vector for $L(\lambda)$.  Then for all $\mu\in X^+$ we have for the double braiding
\[
R_{21}R:L(\lambda)\ot L(\mu)\to L(\lambda)\ot L(\mu),\ \ v_\lambda\ot v_\mu\mapsto q^{-2(\lambda,\mu)}v_\lambda\ot v_\mu+\text{lower degree terms}.
\]
Triviality of this operation demands $2(\lambda,\mu)\in 2l\mbb{Z}$, and hence that $(\lambda,\mu)\in l\mbb{Z}$.  Since this holds for all simples $L(\mu)$ in $\rep G_q$, we find $(\lambda,X^+)\subset l\mbb{Z}$.  Since $X$ is spanned by dominant weights, by Lemma~\ref{lem:401}, we conclude $\lambda\in X^{\M}$.  So we see that all simples in the M\"uger center lie in the image of the $\rep G^\vee$.
\par

Finally, for arbitrary $V$ in the M\"uger center we find that all of its simple composition factors lie in $\rep G^\vee$, since the M\"uger center is closed under subquotients.  As the image of $\rep G^\vee$ is closed under extension in $\rep G_q$, it follows that $V$ is in $\rep G^\vee$.
\end{proof}

\section{Tensor properties and finiteness of $(\rep G_q)_{G^\vee}$}

We begin by recalling the notion of de-equivariantization~\cite{arkhipovgaitsgory03,dgno10}.  We maintain our assumption that the base field is $\mbb{C}$ for consistency, although many of the results are characteristic independent.  By a corepresentation we always mean a {\it right} corepresentation.

\subsection{De-equivariantization and faithful flatness}

Let $\Pi$ be an affine group scheme and $F:\rep \Pi\to \msc{C}$ be a central embedding into a tensor category $\msc{C}$.  That is, $F$ is a pair of an embedding $F_0:\rep\Pi\to \msc{C}$ and a choice of lift to the Drinfeld center $F_1:\rep\Pi\to Z(\msc{C})$.  Such a lift $F_1$ simply specifies a family of half-braidings $\gamma_{V,W}:F_0(V)\ot W\to W\ot F_0(V)$ for objects $V$ in $\rep\Pi$.  This family is required to be natural in $V$.  We abuse notation throughout and write simply $F(V)$ for the image of an object $V$ in $\rep\Pi$ under a central embedding $F$.
\par

The central embeddings of interest to us come from braided tensor functors, in which case the central structure is implicit.  Namely, the braiding on $\msc{C}$ specifies a section $\msc{C}\to Z(\msc{C})$ of the forgetful functor $Z(\msc{C})\to \msc{C}$.  One uses this section to provide any functor into $\msc{C}$ with a canonical central structure.

For any central embedding $F:\rep\Pi\to \msc{C}$ we have the algebra object $F\O=F\O(\Pi)$ in the $\Ind$-category $\Ind\msc{C}$.  We can therefore consider $F\O$-modules in $\Ind\msc{C}$.  Each $F\O$-module becomes a bimodule via the half braiding $\gamma_{\O,-}$.

\begin{definition}
A module $M$ over an algebra object $\msc{A}$ in $\Ind\msc{C}$ is called finitely presented if there are objects $V_0$ and $V_1$ in $\msc{C}$ for which there is an exact sequence $\msc{A}\ot V_1\to \msc{A}\ot V_0\to M$, where the $\msc{A}\ot V_i$ are given the free left $\msc{A}$-action.
\end{definition}

Given a central embedding $F:\rep\Pi\to \msc{C}$, we define the de-equivariantization $\msc{C}_\Pi$ as
\[
\msc{C}_\Pi:=\left\{\text{The category of finitely presented $F\O$-modules in }\Ind\msc{C}\right\}.
\]
This category is naturally additive, enriched over $\mbb{C}$, and monoidal under the tensor product $\ot_{F\O}$ (cf.~\cite{dgno10}).

\begin{definition}
We say a central embedding $F$ is faithfully flat if the resulting de-equivariantization $\msc{C}_\Pi$ is rigid.  We call $F$ locally finite if the de-equivariantization $\msc{C}_\Pi$ is a locally finite category.
\end{definition}

Taken together, $F$ is faithfully flat and locally finite if and only if the de-equivariantization $(\msc{C}_\Pi,\ot_{F\O})$ is a tensor category.  Implicit in our locally finite definition is the proposal that $\msc{C}_\Pi$ is abelian.  Since the de-equivariantization functor $dE:\msc{C}\to \msc{C}_\Pi$, $V\mapsto \O\ot V$, is left adjoint to the forgetful functor $\msc{C}_\Pi\to \Ind \msc{C}$, we see that the forgetful functor is left exact.  It follows that the abelian structure on $\msc{C}_\Pi$ must be the one inherited from $\msc{C}$.  That is to say, $\msc{C}_\Pi$ is abelian if and only if $F\O$ is a coherent algebra in $\Ind \msc{C}$, and local finiteness of $F$ therefore implies coherence of $F\O$ (cf.\ Lemma~\ref{lem:coherence}).

\subsection{Faithful flatness for Hopf inclusions}

Let $\O$ be a commutative Hopf algebra and $\O\to A$ be a Hopf inclusion.  Suppose that this inclusion comes equipped with a function $R:\O\ot A\to \mbb{C}$ which is trivial on $\O\ot \O$ and induces a lift $\corep\O \to Z(\corep A)$ of the corestriction map $\corep\O\to \corep A$.  So, $R$ is a ``half $R$-matrix".  Take $\Pi=\Spec\O$.
\par

For $\corep A$, the $\Ind$-category is simply the category of arbitrary corepresentations $\Corep A$.  We consider the category $_\O\msc{M}^A$ of relative Hopf modules which are finitely presented over $\O$~\cite{montgomery93}.  We have directly $_\O\msc{M}^A=(\corep A)_\Pi$.  If this category is rigid, then the forgetful (monoidal) functor
\[
(\corep A)_\Pi\to (\O\text{-bimod},\ot_\O)
\]
necessarily preserves duals.  Since a bimodule over $\O$ is dualizable if and only if it is projective on the left and on the right, it follows that each object in the de-equivariantization $(\corep A)_\Pi$ is projective over $\O$ in this case.  Conversely, if each object in $(\corep A)_\Pi$ is projective over $\O$ then we have the duals
\begin{equation}\label{eq:duals}
M^\vee =\Hom_{\text{mod-}\O}(M,\O)\ \ \text{and}\ \ {^\vee M}=\Hom_{\O\text{-mod}}(M,\O)
\end{equation}
with actions of the topological Hopf algebra $A^\ast$, i.e.\ $A$-coactions, defined by
\[
f\cdot^l\chi:=\left(m\mapsto f_1\chi(S(f_2)m)\right)\ \ \text{and}\ \ f\cdot^r \chi:=\left(m\mapsto f_1\chi(S^{-1}(f_2)m)\right)
\]
respectively.  The following is basically a result of Masuoka and Wigner.

\begin{lemma}[{\cite[Corollary 2.9]{masuokawigner94}}]\label{lem:487}
Take $K$ to be the coalgebra $\mbb{C}\ot_\O A$ given by taking the fiber at the identity of $\Pi$.  In the above context, the following are equivalent:
\begin{enumerate}
\item[(a)] The category $(\corep A)_\Pi$ is rigid.
\item[(a$'$)] The embedding $F:\rep\Pi\to \corep A$ is faithfully flat.
\item[(b)] The extension $\O\to A$ is faithfully flat.
\item[(c)] Taking the fiber at the identity $\mbb{C}\ot_\O-:(\corep A)_\Pi\to \corep K$ is an equivalence of $\mbb{C}$-linear categories.
\end{enumerate}
In this case $F$ is also locally finite, $A$ is coflat over $K$, and $\O$ is equal to the $K$-coinvariants $\O=A^K$.
\end{lemma}

\begin{proof}
First note that (a) and (a$'$) are equivalent, by definition.  In~\cite{masuokawigner94} the authors employ the category $_\O\mbb{M}^A$ of arbitrary Hopf modules, and prove an infinite analog of the proposed equivalence, with $(\corep A)_\Pi$ replaced with $_\O\mbb{M}^A$ and $\corep K$ replaced with $\Corep K$.  So we are left with the task of translating between the finite and infinite settings.

We have $_\O\mbb{M}^A=\Ind {_\O\msc{M}^A}$ and recover ${_\O\msc{M}^A}$ as the category of compact objects in $_\O\mbb{M}^A$ (cf. Lemma~\ref{lem:1025} below).  One can use this identification to equate (a)--(c) via~\cite[Corollary 2.9]{masuokawigner94}.  Supposing (a)--(c), coflatness of $A$ over $K$ follows by~\cite[Theorem 1]{takeuchi79}, as does the equality $\O=A^K$.  Additionally, $(\corep A)_\Pi$ is locally finite in this case as it is equivalent to the locally finite category $\corep K$, so that $F$ is locally finite by definition.
\end{proof}

\begin{remark}
It is proposed in~\cite[Proposition 3.12]{arkhipovgaitsgory03} that an arbitrary extension $\O\to A$ of a commutative Hopf algebra is faithfully flat.  While the result is correct~\cite{gaitsgory}, there are some problems with the proof given in~\cite{arkhipovgaitsgory03}.  So we have avoided direct reference to this result.
\end{remark}

\subsection{Faithful flatness of the quantum Frobenius}

One can argue exactly as in~\cite[\S 3.9]{arkhipovgaitsgory03}, where some slightly different restrictions on $q$ and $G$ are involved, to find that the linear dual of $u^{\M}_q(G)$ is the fiber $\mbb{C}\ot_\O\O_q$ of the quantum function algebra $\O_q$ at $1\in G^\vee$.  They show further that the quantum Frobenius $\Fr:\rep G^\vee\to \rep G_q$ is, in our language, faithfully flat.  

\begin{theorem}[{\cite[Theorem 2.4]{arkhipovgaitsgory03}}]\label{thm:ag}
The functor $\mbb{C}\ot_\O-:(\rep G_q)_{G^\vee}\to \rep u^{\M}_q(G)$ given by taking the fiber at the identity of $G^\vee$ is a $\mbb{C}$-linear equivalence.
\end{theorem}

We apply Lemma~\ref{lem:487} to obtain

\begin{corollary}\label{cor:ff}
The de-equivariantization $(\rep G_q)_{G^\vee}$, with its natural $\mbb{C}$-enriched monoidal structure $\ot_{\O(G^\vee)}$, is a finite tensor category.
\end{corollary}

\begin{proof}
All is clear save for the finiteness of $(\rep G_q)_{G^\vee}$.  But this just follows from the fact that the equivalent category $\rep u^{\M}_q(G)$ is finite.
\end{proof}

\section{Quasi-fiber functors and the ribbon structure}

We note that the braiding on $\rep G_q$ induces a unique braiding on $(\rep G_q)_{G^\vee}$ so that the de-equivariantization functor $dE:\rep G_q\to (\rep G_q)_{G^\vee}$, $V\mapsto \O\ot V$, is a map of braided tensor categories~\cite[Proposition 4.22]{dgno10}.  This braiding is given simply by 
\[
c_{M,N}:M\ot_\O N\to N\ot_\O M,\ \ m\ot n\mapsto \operatorname{swap}(R\cdot m\ot n).
\]
We consider $(\rep G_q)_{G^\vee}$ as a braided tensor category with this induced braiding throughout the remainder of this document.

\subsection{The ribbon structure on $(\rep G_q)_{G^\vee}$}

We employ the duals~\eqref{eq:duals} to give $(\rep G_q)_{G^\vee}$ an explicit rigid structure.  For $\rho$ the sum of the positive roots, $\rho=\sum_{\gamma\in \Phi^+}\gamma\in X$, the global operator $K_\rho$ provides $\rep G_q$ with a canonical pivotal structure.  Specifically, the natural linear isomorphisms
\[
\operatorname{piv}_V:V\to {V^{\ast}}^\ast,\ \ v\mapsto K_\rho\cdot \mrm{ev}_v,
\]
provide an isomorphism of tensor functors $id\to {(-)^{\ast}}^\ast$.  The pivotal structure on $\rep G_q$ induces a canonical ribbon structure with ribbon element $v=K_\rho u^{-1}$, where $u\in\hatU_q$ is the Drinfeld element~\cite[Corollary 8.3.16]{charipressley95}.

\begin{lemma}\label{lem:ribbon}
When $G$ is simply-connected, or more generally when $K_\rho|_{X^{\M}}\equiv 1$, there is a unique ribbon structure on $(\rep G_q)_{G^\vee}$ so that the de-equivariantization functor from $\rep G_q$ is a map of ribbon categories.
\end{lemma}

\begin{proof}
Supposing such a ribbon structure exists, uniqueness follows from the fact that the de-equivariantization map is surjective.  So we must establish existence.  It suffices to provide a pivotal structure on $(\rep G_q)_{G^\vee}$ so that the de-equivariantization functor $dE$ preserves the pivotal structure.  Such a pivotal structure is given explicitly by
\[
\operatorname{piv}'_M:M\to {M^\vee}^\vee,\ \ m\mapsto K_\rho\cdot \mrm{ev}_m.
\]
The $\operatorname{piv}'_M$ are $\O$-linear as the image of $K_\rho$ in $\hatU^\vee$, which is just the restriction $K_\rho|_{X^{\M}}$, is identically $1$ in this case.  (Otherwise, $\operatorname{piv}'$ twists the $\O$-action by the translation $K_\rho\cdot-$.)  The $\operatorname{piv}'_M$ are isomorphisms because each $M$ is finite and projective over $\O$, and hence reflexive.
\end{proof}

\subsection{Quasi-fiber functors and the ribbon equivalence to $u^{\M}_q(G)$}
\label{sect:fiber}

For an $\O$-bimodule $M$ we let $M_{sym}$ denote the the symmetric $\O$-bimodule with action specified by the left $\O$-action on $M$.

\begin{lemma}
Fix a balancing function $\omega$ for the character lattice of $G$.  For $M$ and $N$ in $(\rep G_q)_{G^\vee}$, the maps
\[
\tilde{T}_{M,N}^\omega:M_{sym}\ot_\O N_{sym}\to M\ot_\O N,\ \ m\ot n\mapsto \omega(\deg m,\deg n) m\ot n,
\]
are well-defined $\O$-linear isomorphisms which are natural in each factor.  Taking the fiber at the identity gives a natural isomorphism
\[
T^\omega_{M,N}: (\mbb{C}\ot_\O M)\ot_\mbb{C}(\mbb{C}\ot_\O N)\to \mbb{C}\ot_\O(M\ot_\O N),\ \ \bar{m}\ot\bar{n}\mapsto \omega(\deg m,\deg n)\overline{m\ot n}.
\]
The natural isomorphism $T^\omega$ provide the reduction $\mbb{C}\ot_\O-:(\rep G_q)_{G^\vee}\to Vect$ with the structure of a quasi-fiber functor $\underline{fib}^\omega:(\rep G_q)_{G^\vee}\to Vect$.
\end{lemma}

\begin{proof}
Note that the reduction $\mbb{C}\ot_\O-:(\rep G_q)_{G^\vee}\to Vect$ is a faithful functor by Theorem~\ref{thm:ag}.  So we need only show that $T^\omega$ is a well-defined quasi-tensor functor to see that it is a quasi-fiber functor.  One simply checks, for $f\in \O$ and $m\ot n\in M\ot_\O N$, the formula
\[
\begin{array}{ll}
\omega(\deg m+\deg f,\deg n)fm\ot n\\
=\omega(\deg m,\deg n)fm\ot n & (\text{balancing property (c)})\\
=q^{-(\deg f,\deg m)}\omega(\deg m,\deg n)m\ot fn\\
=\omega(\deg m,\deg n+\deg f) m\ot fn &(\text{balancing property (b)})
\end{array}
\]
to see that $\tilde{T}^\omega$ provides well-defined, natural, morphisms from the tensor product $M_{sym}\ot_\O N_{sym}$.  The inverse is constructed by a similar use of $\omega$ to see that $\tilde{T}^\omega$ is a natural isomorphism.  The remaining claims of the lemma follow.
\end{proof}

The quasi-fiber functor $\underline{fib}^\omega$ is a linear equivalence onto the subcategory $\rep u^{\M}_q(G)\subset Vect$, by Theorem~\ref{thm:ag}, and hence induces a unique tensor structure on $\rep  u^{\M}_q(G)$ under which the product is the linear tensor product.  As one would expect, this tensor structure is the one introduced in Section~\ref{sect:uM}.

\begin{proposition}\label{prop:qhuM_exists}
Give $u^{\M}_q(G)$ the quasitriangular quasi-Hopf structure provided by a choice of balancing function $\omega$, and give $\rep u^{\M}_q(G)$ the corresponding braided tensor structure.  The functor
\[
fib^\omega:=\{\mbb{C}\ot_\O-,T^\omega\}:(\rep G_q)_{G^\vee}\to \rep u^{\M}_q(G)
\]
is an equivalence of braided tensor categories.  When $K_\rho|_{X^{\M}}\equiv 1$, $fib^\omega$ is additionally and equivalence of ribbon categories.
\end{proposition}

\begin{proof}
We have the diagram
\[
\xymatrix{
\rep G_q=\rep \hatU_q\ar[rr]^{dE}\ar[d]_{\{id,\omega\cdot-\}} & & (\rep G_q)_{G^\vee}\ar[d]^{fib^\omega}\\
\rep \hatU_q^\omega\ar[rr]^{restrict} & & \rep u^{\M}_q,
}
\]
with all but $fib^\omega$ having been established to be braided tensor functors, and ribbon when applicable.  By surjectivity of $dE$ it follows that $fib^\omega$ is a braided tensor functor, and also a ribbon equivalence when applicable, by Theorem~\ref{thm:ag}.
\end{proof}

\section{Rational (de-)equivariantization and non-degeneracy}
\label{sect:actions}

We provide rational analogs of the results of~\cite[Proposition 4.30, Corollary 4.31]{dgno10}.  This section can be seen as an elaboration on the materials of~\cite[\S 2.2]{davydovetingofnikshych18} (cf.~\cite[\S 4.3]{arkhipovgaitsgory03}).  What we need is the following.

\begin{theorem}\label{thm:dgno}
Let $\Pi$ be an affine group scheme.  Suppose that $F:\rep\Pi\to \msc{C}$ is a braided tensor embedding, which is additionally faithfully flat, locally finite, and has M\"uger central image.  Then the de-equivariantization $\msc{C}_\Pi$ is non-degenerate if and only if $F$ is an equivalence onto the M\"uger center of $\msc{C}$.
\end{theorem}

Recall that a braided tensor category $\msc{D}$ is called non-degenerate if its M\"uger center is trivial.  Recall also that a log-modular tensor category is a finite, non-degenerate, ribbon category.  We call a ribbon quasi-Hopf algebra log-modular if its representation category is log-modular.  We observe our calculation of the M\"uger center of $\rep G_q$ at Theorem~\ref{thm:calculatingZ} to arrive at the following.

\begin{corollary}\label{cor:nondegen}
\begin{enumerate}
\item[(a)] The de-equivariantization $(\rep G_q)_{G^\vee}$, with its induced braiding, is non-degenerate.  If furthermore $G$ is simply-connected, then $(\rep G_q)_{G^\vee}$ is canonically log-modular.
\item[(b)] The quasitriangular quasi-Hopf algebra $u^{\M}_q(G)$ is factorizable, and log-modular when $G$ is simply-connected.
\end{enumerate}
\end{corollary}

We are left to prove Theorem~\ref{thm:dgno}.  We have elected to give a completely general presentation of (de-)equivariantization for tensor categories, in order to make precise sense of the conjectural relations with vertex operator algebras discussed in Section~\ref{sect:(1,p)}.  However, to keep from distracting completely from our main program, we defer many of the details to Appendix~\ref{sect:A}.

\subsection{Rational actions on cocomplete categories}
\label{sect:cocomplete}

Let $\msc{D}$ be a {\it cocomplete} $\mbb{C}$-linear category.  For any commutative algebra $S$ we let $\msc{D}_S$ denote the $S$-linear category consisting of objects $X$ in $\msc{D}$ equipped with an $S$-action, $S\to \End_\msc{D}(X)$.  Maps in $\msc{D}_S$ are maps in $\msc{D}$ which commute with the $S$-action.  We note that this operation $(?)_S$ is functorial in $\mbb{C}$-linear morphisms, so that a $\mbb{C}$-linear morphism $\msc{D}\to \msc{D}'$ induces a $S$-linear morphism $\msc{D}_S\to \msc{D}'_S$.  If we have an algebra map $k:S\to T$ we restrict scalars to get a map of linear categories $k^\ast:\msc{D}_T\to \msc{D}_S$.

Restriction has a left adjoint $k_\ast:\msc{D}_S\to \msc{D}_T$ given by induction.  Here we use cocompleteness of $\msc{D}$ to construct the induction $T\ot_SX$ explicitly as the quotient of the sum $\oplus_{a\in T}Xa$ by the standard relations, where $Xa$ is just a copy of $X$ labeled by $a\in T$.
\par

Let $\Pi$ be an affine group scheme with algebra of functions $R=\O(\Pi)$.  A rational action of $\Pi$ on $\msc{D}$, or simply an ``action", consists of the following information:
\begin{enumerate}
\item[(a)] A functor $\psi_u:\msc{D}\to \msc{D}_R$ which is exact and commutes with colimits.
\item[(b)] A choice of coassociative isomorphism $\sigma:\Delta_\ast \psi_u\overset{\sim}\to \psi_u\psi_u$ of functors from $\msc{D}$ to $\msc{D}_{R\ot R}$.
\item[(c)] A choice of isomorphism $\eta:\epsilon_\ast\psi_u\overset{\sim}\to id_\msc{D}$ for the counit $\epsilon:R\to \mbb{C}$.
\end{enumerate}
Given $\msc{D}$ with an action of $\Pi$ we define the category of equivariant objects $\msc{D}^\Pi$ as the non-full subcategory of objects $X$ in $\msc{D}$ equipped with a coaction $\rho_X:X\to \psi_u X$ which is coassociative and counital, in the sense of the equalities
\[
\psi_u(\rho_X)\rho_X=\sigma_X \Delta_\ast\rho_X\ \ \mrm{and}\ \ \eta_X\epsilon_\ast\rho_X=id_X.
\]
Morphisms of equivariant objects are maps $f:X\to Y$ in $\msc{D}$ for which the diagram
\[
\xymatrix{
\psi_u X\ar[rr]^{\psi_u f} & & \psi_u Y\\
X\ar[rr]^f\ar[u]^{\rho_X} & & Y\ar[u]_{\rho_Y}
}
\]
commutes.
\par

Note that for $\msc{D}$ with a $\Pi$-action we can change base along $S$-points $t\in \Pi(S)$, $t:\Spec(S)\to \Pi$, to obtain a compatible collection of maps $\psi_t:\msc{D}\to \msc{D}_S$.  These maps have induced compatible isomorphisms $\psi_t\psi_{t'}\cong \psi_{t\cdot t'}$, where for points $t\in \Pi(S)$ and $t'\in\Pi(S')$ we let $t\cdot t'=(t\ot t')\Delta$ denote the product in $\Pi(S\ot S')$.  In particular each element in the discrete group $x\in \Pi(\mbb{C})$ acts via an equivalence $\psi_x:\msc{D}\to \msc{D}$, and we recover from the rational action of $\Pi$ an action of the discrete group $\Pi(\mbb{C})$ on $\msc{D}$, in the usual sense of~\cite{dgno10}.

\begin{remark}
Our presentation of rational group actions on categories is adapted from informal notes of D.\ Gaitsgory.
\end{remark}

\subsection{Rational group actions on tensor categories}

A locally finite category $\msc{D}$ is explicitly not cocomplete, as all objects are required to be of finite length.  In this case we define $\msc{D}_S$ only for {\it coherent} $S$, as the full subcategory of objects in $(\Ind\msc{D})_S$ with a finite presentation $unit_\ast V\to unit_\ast W\to X$, where the $V$ and $W$ are in $\msc{D}$ and $unit_\ast:\Ind\msc{D}\to (\Ind\msc{D})_S$ in induction by the unit $\mbb{C}\to S$.  As a more practical check for finite presentation we have

\begin{lemma}\label{lem:1025}
The subcategory $\msc{D}_S\subset (\Ind\msc{D})_S$ is exactly the subcategory of compact objects in $(\Ind\msc{D})_S$.
\end{lemma}

We provide a proof of the lemma in Appendix~\ref{sect:A}.  We employ these categories $\msc{D}_S$ and define a $\Pi$-action on $\msc{D}$ just as above, and also the category $\msc{D}^\Pi$ of equivariant objects.  (Recall that the algebra of functions on an affine group scheme is itself coherent, by Lemma~\ref{lem:coherence}.)
\par

When $\msc{D}$ is a finite tensor category each $\msc{D}_S$ is monoidal under the product $X\ot_S Y$, which is given as the quotient of the product $X\ot Y$ internal to $\msc{D}$ by the relations $s\ot 1-1\ot s:X\ot Y\to X\ot Y$, for each $s\in S$.  We say $\Pi$ acts on $\msc{D}$, as a tensor category, if the universal map $\psi_u:\msc{D}\to \msc{D}_R$ is equipped with a monoidal structure $\psi_u(V)\ot_R\psi_u(W)\cong \psi_u(V\ot W)$ which is compatible with the isomorphism $\sigma$, in the sense that the two paths from $\psi_u(V)\ot_R\psi_u(W)$ to $\psi_u\psi_u(V\ot W)$ agree.  This implies that for each $S$-point $t\in \Pi(S)$ the induced maps $\psi_t:\msc{D}\to \msc{D}_S$ will all be monoidal functors in a compatible manner.

\begin{lemma}
When $\msc{D}$ is a tensor category, any monoidal functor $\psi_u:\msc{D}\to (\Ind\msc{D})_R$ has image in $\msc{D}_R$, and hence $\psi_u$ defines a rational action $\Pi$, provided $\psi_u$ is exact and commutes with colimits.
\end{lemma}

\begin{proof}
Monoidal functors preserve dualizable objects, and dualizable objects are compact.
\end{proof}

When $\msc{D}$ is braided, the base change $\msc{D}_S$ additionally admits a unique braiding so that the induction functor $unit_\ast:\msc{D}\to \msc{D}_S$ is a braided tensor functor.  Whence $\Pi$ can act on $\msc{D}$ as a braided tensor category, in which case the action map $\psi_u:\msc{D}\to \msc{D}_S$ is assumed to be a braided monoidal functor. 

For a (braided) tensor category $\msc{D}$ equipped with a $\Pi$ action, which respects the (braided) tensor structure, the equivariantization $\msc{D}^\Pi$ is a non-full (braided) tensor subcategory in $\msc{D}$.  The coaction on a product $V\ot W$ of equivariant objects is simply given by the composite $V\ot W\overset{\rho_V\ot\rho_W}\to \psi_u V\ot_R \psi_u W\cong \psi_u(V\ot W)$.

\subsection{A summary of the details in Appendix~\ref{sect:A}}

Fix $\msc{C}$ a tensor category with a faithfully flat, locally finite, central embedding $F:\rep\Pi\to \msc{C}$.  Fix also a tensor category $\msc{D}$ with a rational action of $\Pi$.  There is a canonical $\Pi$-action on the de-equivariantization $\msc{C}_\Pi$, given by the formula $\psi_u(X):=R\ot X$, and an obvious functor
\[
\operatorname{can}^!:\msc{C}\overset{\sim}\to (\msc{C}_\Pi)^\Pi,\ \ V\mapsto \O\ot V,
\]
which is shown to be a tensor equivalence at Proposition~\ref{prop:EqdEq}.  Similarly, there is a canonical central embedding into the de-equivariantization $\rep\Pi\to \msc{D}^\Pi$ and an equivalence
\[
\operatorname{can}_!:\msc{D}\overset{\sim}\to (\msc{D}^\Pi)_\Pi,\ \ W\mapsto \psi_u(W),
\]
as verified in Proposition~\ref{prop:dEqEq}.

Suppose now that $\msc{C}$ is braided and that $\rep\Pi\to \msc{C}$ has M\"uger central image.  Suppose additionally that $\msc{D}$ is braided and that the action of $\Pi$ respects the braiding.  We say a tensor subcategory $\msc{W}\subset \msc{D}$ is $\Pi$-stable if the restriction of the action functor $\psi_u:\msc{W}\to \msc{D}_R$ has image in $\msc{W}_R$.  For such $\Pi$-stable $\msc{W}$ we have an induced inclusion of the equivariantizations $\msc{W}^\Pi\subset \msc{D}^\Pi$.
\par

Similarly, for any intermediate tensor subcategory $\rep\Pi\to \msc{K}\to \msc{C}$ we have an inclusion of the de-equivariantization $\msc{K}_\Pi\to \msc{C}_\Pi$.  Since $\msc{C}_\Pi$ is abelian $F\O$ is coherent in $\msc{C}$, and hence in $\msc{K}$ as well.  So $\msc{K}_\Pi$ is abelian.  Local finiteness of $\msc{C}_\Pi$ also implies local finiteness of $\msc{K}_\Pi$, and the fact that the duals of free objects in $\msc{K}_\Pi$ remain in $\msc{K}_\Pi$ implies, by considering presentations, that the duals of all object in $\msc{K}_\Pi$ remain in $\msc{K}_\Pi$.  So the intermediate inclusion $\rep\Pi\to \msc{K}$ is faithfully flat and locally finite as well, and $\msc{K}_\Pi$ is a tensor subcategory in $\msc{C}_\Pi$.
\par

One can deduce from obvious naturality properties of the equivalences $\operatorname{can}^!$ and $\operatorname{can}_!$ the following proposition, just as in~\cite{dgno10}.

\begin{proposition}[{cf.~\cite[Proposition 4.30]{dgno10}}]\label{prop:885}
De-/equivariantization provides a bijection between the poset of isomorphism-closed intermediate tensor subcategories $\rep\Pi\to \msc{K}\to\msc{C}$ and isomorphism-closed $\Pi$-stable tensor subcategories $\msc{W}\to \msc{C}_\Pi$.  This bijection restricts to a bijection for braided (resp.\ M\"uger central) intermediate categories in $\msc{C}$ and $\Pi$-stable braided (resp.\ M\"uger central) subcategories in $\msc{C}_\Pi$.
\end{proposition}

We prove Proposition~\ref{prop:885} in Section~\ref{sect:prop_proof}.

\subsection{Proof of Theorem~\ref{thm:dgno} and Corollary~\ref{cor:nondegen} from Proposition~\ref{prop:885}}

\begin{proof}[Proof of Theorem~\ref{thm:dgno}]
Suppose that $F:\rep\Pi\to \msc{C}$ is an equivalence onto the M\"uger center of $\msc{C}$.  Then for any intermediate M\"uger central category $\rep\Pi\to \msc{K}\to \msc{C}$ the map $\rep\Pi\to \msc{K}$ is an equivalence.  By Proposition~\ref{prop:885} it follows that for any M\"uger central subcategory $\msc{W}$ in $\msc{C}_\Pi$ the inclusion $Vect\subset \msc{W}$ is an equivalence.  So the M\"uger center of $\msc{C}_\Pi$ is trivial, and by definition $\msc{C}_\Pi$ is non-degenerate.
\par

Conversely, if the M\"uger center of $\msc{D}=\msc{C}_\Pi$ is trivial then we apply Proposition~\ref{prop:885} again to find that for any central intermediate category $\rep\Pi\to \msc{K}\to \msc{C}$ the inclusion from $\rep\Pi$ to $\msc{K}$ is an equivalence.  This holds in the particular case in which $\msc{K}$ is the M\"uger center of $\msc{C}$, so that $F$ is seen to be an equivalence onto the M\"uger center of $\msc{C}$.
\end{proof}

\begin{proof}[Proof of Corollary~\ref{cor:nondegen}]
(a) We already understand that $(\rep G_q)_{G^\vee}$ is finite, braided, and ribbon when $G$ is simply-connected, by Corollary~\ref{cor:ff} and Lemma~\ref{lem:ribbon}.  So we need only establish non-degeneracy.  But this follows immediately by Theorem~\ref{thm:calculatingZ} and Theorem~\ref{thm:dgno}.  Statement (b) follows from (a) and Proposition~\ref{prop:qhuM_exists}.
\end{proof}

\section{Revisiting the odd order case}
\label{sect:odd}

Let $\xi$ be an odd order root of unity, and take $\ell=\ord(\xi)$.  We return to the odd order case to clarify the appearance of adjoint type groups in certain constructions related to $u_\xi(\mfk{g})$ (e.g.~\cite{davydovetingofnikshych18}).  Here we have $u_\xi(\mfk{g})$ as the {\it Hopf} subalgebra in the usual divided power algebra $U_\xi(\mfk{g})$ generated by the $E_\alpha$, $F_\alpha$, and $K_\alpha$ (with $K_\alpha^\ell=1$).

\subsection{Construction of $\rep u_\xi(\mfk{g})$ from $\rep G_\xi$}

We only sketch the details, as the situation is actually quite a bit easier to deal with than in the even order case.
\par

Let $G$ be of {\it adjoint} type with Lie algebra $\mfk{g}$.  Suppose $\ell$ is coprime to the determinant of the Cartan matrix for $\mfk{g}$ and also the $d_i$ (as is a standard assumption).  This implies that the form on the quotient $Q/\ell Q=G(u_\xi)^\vee$ induced by the Killing form is non-degenerate.  So we see that $Q^{\M}=\ell Q$ in this case, and the quantum Frobenius $Fr:\rep G\to \rep G_\xi$, which in this case involves no duality for $G$, is an equivalence onto the M\"uger center.  (One verifies this just as in Theorem~\ref{thm:calculatingZ}.)  So the de-equivariantization $(\rep G_\xi)_G$ is non-degenerate, and in fact log-modular, by Theorem~\ref{thm:dgno}.
\par

Now, in this case, the quantum Frobenius is associated to a Hopf inclusion $Fr:\O(G)\to \O_\xi(G)$ with {\it central} image, and for which the restrictions of the $R$-matrix to $\O\ot \O_\xi$ and $\O_\xi\ot \O$ is identically $1$.  Taking the fiber then provides a linear equivalence
\[
\mbb{C}\ot_\O-:(\rep G_\xi)_G\to \rep u_\xi(\mfk{g}),
\]
which is furthermore seen to be a {\it braided tensor} equivalence, via the strong centrality properties of the quantum Frobenius.  So we see that the construction of the standard small quantum group at a root of unity of odd order is essentially an {\it adjoint} type construction, as opposed to a simply-connected construction.
\par

The above presentation is given in contrast to the original presentation of the quantum Frobenius~\cite{lusztig89,lusztig90,lusztig90II}, which suggests that the small quantum group is principally a simply-connected object.  (Indeed, one can construct the small quantum group from the simply-connected form of $G$, via the original quantum Frobenius~\cite[Theorem 7.2]{deconcinilyubashenko94}.)

\begin{remark}
Our comment here is specifically about the standard choice of grouplikes for $u_\xi(\mfk{g})$ at odd order parameter.  Namely, the choice of the grouplikes as the elementary abelian $\ell$-group generated by the $K_\alpha$.  One can, of course, construct $u_\xi(G)$ at arbitrary $G$ and $\xi$ in accordance to the processes outlined in the present work.  We would propose, however, that the grouplikes should vary in a meaningful way with the choice of $G$ and $\xi$.
\end{remark}

\section{Identifications with quantum groups of Creutzig et al.\ and Gainutdinov et al.}
\label{sect:id}

We clarify that all current means of producing log-modular quantum groups at even order roots of unity agree (at the ribbon categorical level).  In particular, we identify our quasi-Hopf algebras with those of~\cite{creutzigetal,gainutdinovlentnerohrmann}.  We also provide a brief discussion of the remarkable nature of small quantum $\PSL_2$, particularly at $q=e^{\pi i/4}$.

\subsection{Toral construction of the log-modular kernel}

Let $\dotu_q=\dotu_q(G)$ be the subalgebra in $\dotU_q$ generated by the idempotents $1_\lambda$, $\lambda\in X$, and the elements $E_\alpha$, $F_\alpha$.  The category $\rep \dotu_q$ is a tensor category and we have the restriction functor $\rep G_q=\rep \dotU_q\to \rep \dotu_q$.  The $R$-matrix for $\rep G_q$ restricts to a global operator for $\dotu_q$, as does the pivotal element $K_\rho$, and $\rep\dotu_q$ is therefore ribbon.
\par

The quantum Frobenius for $\dotU_q$ restricted to $\dotu_q$ has image equal to the (non-unital) subalgebra $\mbb{C}[1_\mu:\mu\in X^{\M}]$ in $\dotU^\vee$.  Hence the quantum Frobenius restricts to a M\"uger central tensor functor $\rep T^\vee\to \rep\dotu_q$.  We can consider now the de-equivariantization $(\rep\dotu_q)_{T^\vee}$, and the map $(\rep\dotu_q)_{T^\vee}\to \rep u^{\M}_q(G)$ given by taking the fiber at the identity of $T^\vee$.  Note that we have a diagram of $\mbb{C}$-linear functors
\[
\xymatrix{
(\rep G_q)_{G^\vee}\ar[dr]_{\O(T^\vee)\ot_{\O(G^\vee)}}\ar[rr]^{\mbb{C}\ot_{\O(G^\vee)}} & & \rep u^{\M}_q\\
 & (\rep\dotu_q)_{T^\vee}\ar[ur]_{\mbb{C}\ot_{\O(T^\vee)}}
}
\]

\begin{proposition}\label{prop:894}
The functor $\mbb{C}\ot_{\O(T^\vee)}-:(\rep \dotu_q)_{T^\vee}\to \rep u^{\M}_q(G)$ is a $\mbb{C}$-linear equivalence, and becomes a braided tensor equivalence with the tensor compatibility $T^\omega$ as in Proposition~\ref{prop:qhuM_exists}.  In the simply-connected case $\mbb{C}\ot_{\O(T^\vee)}-$ is furthermore a ribbon equivalence.
\end{proposition}

\begin{proof}
The result at the abelian level appears in~\cite[Proof of Theorem 4.7]{arkhipovgaitsgory03}.  The tensor structure, and ribbon structure, are dealt with in exactly the same manner as in Proposition~\ref{prop:qhuM_exists}.
\end{proof}

\subsection{Identification with the log-modular quantum group of Creutzig et al.~\cite{creutzigetal}}

Take $u^{\M}_q(\mfk{sl}_2)$ to be the simply-connected form $u_q^{\M}(\SL_2)$.  In~\cite{gainutdinovrunkel17,creutzigetal} the authors construct a log-modular quasi-Hopf algebra $u^{\phi}_q(\mfk{sl}_2)$ via local modules over an algebra $\Lambda$ in the braided tensor category of (weight graded) representations of the unrolled quantum group $\rep_{wt}u^H_q(\mfk{sl}_2)$.  The category $\rep_{wt}u^H_q(\mfk{g})$ is the category of $\mbb{C}=X_\mbb{C}$-graded vector spaces with actions of operators $E$ and $F$ which shift the grading appropriately and satisfying the usual relations of the quantum group.  Since $\rep \dotu_q(\mfk{sl}_2)$ is the category of $X=\mbb{Z}[\frac{1}{2}\alpha]$-graded vector spaces with corresponding actions of $E$ and $F$, we see that there is a tensor embedding
\begin{equation}\label{eq:1159}
\rep\dotu_q(\mfk{sl}_2)\to \rep_{wt}u^H_q(\mfk{sl}_2).
\end{equation}
The algebra $\Lambda$ of~\cite{creutzigetal} is the sum of all invertible representations supported on $X^{\M}=lQ$, and is therefore identified with $\O(T^\vee)$ under the map~\eqref{eq:1159}.  Furthermore, since all indecomposable components of $\Lambda=\O(T^\vee)$ are invertible, any local module over $\Lambda$ in $\rep_{wt}u^H_q(\mfk{sl}_2)$ must in fact centralize $\Lambda$.

\begin{proposition}[{\cite[Proposition 3.8]{creutzigetal}}]
The centralizer of $\Lambda=\O(T^\vee)$ in $\rep_{wt}u^H_q(\mfk{sl}_2)$ is equal to $\rep\dotu_q(\mfk{sl}_2)$.
\end{proposition}

The authors show further that there is an equivalence of categories between local, finitely generated, modules over $\Lambda$ in $\rep_{wt}u^H_q(\mfk{sl}_2)$ and $\rep u^\phi_q(\mfk{sl}_2)$.  Since $\Lambda=\O(T^\vee)$ is Noetherian, this is the same as the category of finitely presented local $\Lambda$-modules in $\rep_{wt}u^H_q(\mfk{sl}_2)$, and by the above proposition we find

\begin{theorem}[{\cite[Theorem 4.1]{creutzigetal}}]
There is an equivalence of ribbon categories $(\rep \dotu_q(\mfk{sl}_2))_{T^\vee}\simeq \rep u^\phi_q(\mfk{sl}_2)$.
\end{theorem}

Whence we have the following.

\begin{corollary}\label{cor:cea}
There is an equivalence of ribbon categories $\rep u^{\M}_q(\mfk{sl}_2)\simeq\rep u^\phi_q(\mfk{sl}_2)$.
\end{corollary}

\begin{proof}
Apply Proposition~\ref{prop:894} and~\cite[Theorem 4.1]{creutzigetal}.
\end{proof}

\begin{remark}\label{rem:1}
To be precise, Creutzig, Gainutdinov, and Runkel employ an $R$-matrix of the form $\Omega R^+$, as opposed to $R^+\Omega^{-1}$.  This distinction is, however, utterly unimportant.  Specifically, the choice does no change the M\"uger center of $\rep G_q$, the definition of $(\rep G_q)_{G^\vee}$ as a tensor category, or the definition of $u^{\M}_q(G)$ as a quasi-Hopf algebra.  One simply has to change the $R$-matrix for $u^{\M}_q(G)$ by replacing our $R$ for $\dotu_q$ with the $R$-matrix from~\cite{creutzigetal}, in the most na\"ive manner.
\end{remark}

\subsection{Identification of the log-modular quantum groups of Gainutdinov et al.~\cite{gainutdinovlentnerohrmann}}

In~\cite{gainutdinovlentnerohrmann}, Gainutdinov, Lentner, and Ohrmann construct factorizable quantum groups $u_q(\mfk{g},X)$ for pairs of a simple Lie algebra $\mfk{g}$ and choice of character lattice $X$.  (This is the same as a choice of almost simple algebraic group $G$.)  The $u_q(\mfk{g},X)$ generalize the quantum groups $u^\phi_q(\mfk{sl}_2)$ of~\cite{gainutdinovrunkel17,creutzigetal}.  Their construction is actually more general, and allows for $\mfk{g}$ to be a Lie super-algebra for example.
\par

Let $Y\subset X$ be the Kernel of the killing form $\Omega:X\times X\to \mbb{C}^\times$.  We have $Y\subset X^{\M}$, and the inclusion is generally not an equality.  For example, for $\SL_2$ (or any simply-connected group), $Y=2lQ$ while $X^{\M}=lQ$.  We take $\mbb{T}:=\Spec(\mbb{C}[Y])$, and have the corresponding finite covering $T^\vee\to \mbb{T}$.  Take also $\doto_q$ the finite dual $(\dotu_q)^\circ$.  It follows by Proposition~\ref{prop:894} and Lemma~\ref{lem:487} that $\doto_q$ is faithfully flat over $\O(T)$, and $\O(T)$ is faithfully flat over $\O(\mbb{T})$~\cite[Theorem 3.1]{takeuchi72}, so that $\doto_q$ is faithfully flat over $\O(\mbb{T})$ via the quantum Frobenius.  Subsequently, taking the fiber at the identity provides a braided {\it tensor} equivalence
\begin{equation}\label{eq:943}
\mbb{C}\ot_{\O(\mbb{T})}-:(\rep\dotu_q)_{\mbb{T}}\overset{\sim}\to \rep \dotu_q(\mfk{g},X/Y),
\end{equation}
where $\dotu_q(\mfk{g},X/Y)$ is the finite dimensional quasitriangular Hopf subalgebra in the cofinite completion $\hatu_q$ generated by the character group $\mbb{C}[(X/Y)^\vee]\subset \Fun(X,\mbb{C})\subset \hatu_q$ and the operators $E_\alpha$ and $F_\alpha$.  (See e.g.~\cite[Proposition 4.1]{angionogalindopereira14}.)  This Hopf algebra is furthermore ribbon when $K_\rho|_Y\equiv 1$.
\par

The equivalence~\eqref{eq:943} sends the algebra $\O(T^\vee)$ in $\rep\dotu_q$ to $\mbb{C}[X^{\M}/Y]$, the algebra of functions on the kernel of the projection $T^\vee\to \mbb{T}$.  So the equivalence~\eqref{eq:943} restricts to a braided equivalence
\begin{equation}\label{eq:950}
\mbb{C}\ot_{\O(\mbb{T})}-:(\rep\dotu_q)_{T^\vee}\overset{\sim}\to (\rep\dotu_q(\mfk{g},X/Y))_{X^{\M}/Y}.
\end{equation}
By direct considerations of the definitions, both equivalences~\eqref{eq:943} and~\eqref{eq:950} are equivalences of ribbon categories in the simply-connected case.

\begin{proposition}\label{prop:gea}
There is an equivalence of braided categories $\rep u^{\M}_q(G)\overset{\sim}\to \rep u_q(\mfk{g},X)$, which is additionally a ribbon equivalence at the simply-connected lattice.
\end{proposition}

\begin{proof}
It is shown in~\cite[Theorem 6.7]{gainutdinovlentnerohrmann} that $\rep u_q(\mfk{g},X)$ can be recovered as the de-equivariantization (modularization) $(\rep\dotu_q(\mfk{g},X/Y))_{X^{\M}/Y}$.  So the result follows by the equivalence~\eqref{eq:950} and Proposition~\ref{prop:894}.
\end{proof}

\begin{remark}
As was the case in Remark~\ref{rem:1}, there is an inconsequential difference in the $R$-matrices employed in~\cite{gainutdinovlentnerohrmann} and in the present study.
\end{remark}

\subsection{Some remarks on small quantum $\PSL_2$}
\label{sect:remPSL2}

Recall, from Lemma~\ref{lem:417}, that we have a non-degenerate kernel for $(\PSL_2)_q$ exactly when $q$ is a $2l$-th root of $1$ with $l$ odd or divisible by $4$.  Let us consider the case $4\mid l$.  As usual, take $P$ and $Q$ to be the weight and root lattices for $\mfk{sl}_2$ respectively, and recall $P=\frac{1}{2}Q$.
\par

We can consider the torus forms $\dotu_q(\SL_2)$ and $\dotu_q(\PSL_2)$, and the braided embedding $\rep\dotu_q(\PSL_2)\to \rep\dotu_q(\SL_2)$.  The M\"uger center of $\rep\dotu_q(\SL_2)$ is the subcategory $Vect_{lQ}$ of $lQ$-graded vector spaces, while that of $\rep\dotu_q(\PSL_2)$ is $Vect_{lP}$.  So we have the invertible simple $L(l\alpha/2)$ in $\rep \dotu_q(\SL_2)$ which descends to a simple $\chi=\bar{L}(l\alpha/2)$ in the log-modular kernel $\rep u^{\M}_q(\SL_2)$.  This simple squares to the identity and has centralizer equal to the image of $\rep(\PSL_2)_q$ in $\rep u^{\M}_q(\SL_2)$.  Indeed, the subcategory generated by $\chi$ in $\rep u^{\M}_q(\SL_2)$ is exactly the image of $\rep\SL_2$ in $\rep u^{\M}_q(\SL_2)$.  Hence small quantum $\PSL_2$ is identified with the de-equivariantization of the centralizer of $\chi$ in $\rep u^{\M}_q(\SL_2)$ by the copy of $\rep \mbb{Z}/2\mbb{Z}$ generated by $\chi$,
\[
\rep u_q^{\M}(\PSL_2)\cong (\langle \chi\rangle')_{\langle \chi\rangle}.
\]
By the remarks following Proposition~\ref{prop:682}, we see that the ribbon structure on $\rep u^{\M}_q(\SL_2)$ does not induce a ribbon structure on $\rep u^{\M}_q(\PSL_2)$.

In addition to this relationship with quantum $\SL_2$, $\rep u^{\M}_q(\PSL_2)$ has another remarkable property.  As is explained in Section~\ref{sect:tensor_gen} below, simples in $\rep u^{\M}_q(\PSL_2)$ are in bijection with characters of the group $Q/lP$.  When $l=4$, $Q/4P=Q/2Q$ and we see that $\rep u^{\M}_{e^{\pi i/4}}(\PSL_2)$ has exactly two simples.  One can see directly that that the unique non-trivial simple in $\rep u^{\M}_{e^{\pi i/4}}(\PSL_2)$ is of dimension $2$, and hence non-invertible.  As far as we understand, $\rep u^{\M}_{e^{\pi i/4}}(\PSL_2)$ is the only known non-degenerate finite tensor category with two simples, one of which is non-invertible.

\section{Relations between quantum groups and $(1,p)$ vertex operator algebras}
\label{sect:(1,p)}

For historical reasons we replace $l$ with $p$ in our notation, and take $q$ to be a root of unit of even order $2p$.

\subsection{Tensor generation of $\rep u^{\M}_q(G)$ and $\rep G_q$}
\label{sect:tensor_gen}

Note that any $u^{\M}_q(G)$-representation $V$ decomposes into character spaces $\oplus_{z\in Z} V_z$ for the action of the grouplikes $\mbb{C}[Z^\vee]$.  Since $V$ contains a simple representation for the non-negative subalgebra $u^{\M}_{\geq 0}$, and the Jacobson radical of $u^{\M}_{\geq 0}$ is generated by the $\mathsf{E}_i$, we see that any representation $V$ contains a highest weight vector.
\par

For any element $z\in Z=(Z^\vee)^\vee$ we have the Verma module $M(z)$, and the unique simple quotient $L(z)$, constructed in the standard manner.  Hence we have a bijection between characters for the grouplikes and simples for $u^{\M}_q$, $z\mapsto L(z)$.  The simple $L(z)$ has unique highest weight $z$.

\begin{lemma}\label{lem:926}
The category $\rep u^{\M}_q(G)$ is tensor generated by the simples $\{L(z):z\in Z\}$.
\end{lemma}

\begin{proof}
Note that since the associator $\phi$ for $u^{\M}_q$ lies in the coradical $(u^{\M}_q)_0=\mbb{C}[Z^\vee]$, we can define a coradical filtration for $u^{\M}_q$ recursively via the wedge construction
\[
(u^{\M}_q)_{n+1}:=\ker\left(u^{\M}_q\overset{\nabla}\to u^{\M}_q\ot u^{\M}_q\to \frac{u^{\M}_q}{(u^{\M}_q)_0}\ot\frac{u^{\M}_q}{(u^{\M}_q)_n} \right).  
\]
This resulting filtration is exhaustive and $\nabla(u_n^{\M})=\sum_{i+j=n}u_i^{\M}\ot u_j^{\M}$.
\par

Let $\msc{D}\subset \rep u^{\M}_q$ be the subcategory tensor generated by the simples.  By Tannakian reconstruction $\msc{D}$ is representations of a quotient quasi-Hopf algebra $K$ of $u^{\M}_q$, and the inclusion $\msc{D}\to \rep u^{\M}_q$ is given by restricting along the quotient $u^{\M}_q\to K$.  Indeed, $K$ is the quotient of $u^{\M}_q$ by the collective annihilators of arbitrary products of simples $L(z_1)\ot\dots\ot L(z_r)$.

By considering the simples of $u_q(\mfk{sl}_2)$ we see that for each $\alpha$ there is a simple $L(z_i)$ on which $E_\alpha$ acts non-trivially.  Hence the space of primitives maps injectiviely into the endomorphism ring of the sum of simples $\End_\mbb{C}(\oplus_{z\in Z} L(z))$, via the representation map $u_q^{\M}\to \End_\mbb{C}(\oplus_z L(z))$.  Indeed, the representation map restricts to an injection on the $1$-st component of the coradical filtration $(u^{\M}_q)_1\to \End_\mbb{C}(\oplus_z L(z))$.  So we see that the quasi-Hopf quotient $u^{\M}_q\to K$ is injective on $(u^{\M}_q)_1$.  It follows by induction, and by considering the composite $u^{\M}_q\overset{\nabla}\to u^{\M}_q\ot u^{\M}_q\to K/K_0\ot K/K_0$, that the quotient $u^{\M}_q\to K$ is injective and therefore an isomorphism~\cite[Theorem 5.3.1]{montgomery93}.
\end{proof}

One can alternatively prove Lemma~\ref{lem:926} in the simply-connected setting by noting that $\rep u^{\M}_q(G)$ admits a simple projective object~\cite{gainutdinovrunkelP}.

\begin{lemma}
The category $\rep G_q$ is tensor generated by the simples $\{L(\lambda):\lambda\in X^+\}$.
\end{lemma}

\begin{proof}
Let $\msc{K}$ be the tensor subcategory generated by the simples in $\rep G_q$.  Since the M\"uger center $\rep G^\vee$ is generated by its simples we see that the quantum Frobenius has image in $\msc{K}\subset \rep G_q$.  Since every object in $\rep u^{\M}_q$ is seen to be the quotient of an object from $\rep G_q$, via finite presentation of objects in the equivalent category $(\rep G_q)_{G^\vee}$, for example, it follows that every simple in $\rep u^{\M}_q$ is the quotient of a simple from $\rep G_q$.  Hence the functor $\msc{K}\to \rep u^{\M}_q$ has all of the simples for $u^{\M}_q$ in its image, and by Lemma~\ref{lem:926} this map is therefore surjective.  It follows that the de-equivariantization $\msc{K}_{G^\vee}$, which is an embedded tensor subcategory in $(\rep G_q)_{G^\vee}$, is mapped isomorphically to $\rep u^{\M}_q$ under the fiber $\mbb{C}\ot_\O-:\msc{K}_{G^\vee}\to \rep u^{\M}_q$.  So we see that the inclusion $\msc{K}\to \rep G_q$ is an isomorphism, by Proposition~\ref{prop:885}.
\end{proof}

\subsection{Rephrasing a conjecture of Bushlanov et al.: representations of the $(1,p)$-log minimal model}
\label{sect:(1,p)2} 

Let $\msc{C}_p$ denote the subcategory of $\rep U_q(\mfk{sl}_2)$ generated by the simples.  In~\cite{bfgt09} the authors explain that the category of representations for the divided power algebra $\msc{C}_p$ admits a $\mbb{Z}/2\mbb{Z}$-grading
\[
\msc{C}_p=\msc{C}^+_p\oplus \msc{C}^-_p,
\]
and they conjecture a tensor equivalence between $\msc{C}^+_p$ and the $(1,p)$-Virasoro logarithmic minimal model.  More specifically, if we let $\mcl{L}_p=L(c_p,0)$ denote the (simple but non-rational) Virasoro vertex operator algebra at central charge $c_p=1-6(p-1)^2/p$, they conjecture an equivalence between $\msc{C}^+_p$ and the full subcategory $\rep \mcl{LM}(1,p)$ of $\rep \mcl{L}_p$ additively generated by the indecomposable representations appearing in the $(1,p)$-logarithmic minimal model $\mcl{LM}(1,p)$~\cite{pearceetal06,rasmussenpearce07,rasmussen11}\cite[Eq.\ 1.1]{bfgt09}.

\begin{remark}
The inclusion $\msc{C}_p\to \rep U_q(\mfk{sl}_2)$ is presumably an equality, by the classification of indecomposables for $U_q(\mfk{sl}_2)$~\cite{bgt12}.  The analogous result should hold outside of type $A_1$ by an analysis similar to~\cite[Theorem 9.12]{andersenpolowen91}.
\end{remark}

There is a distinguished invertible simple $\chi=\mbb{C}v$ for $U_q(\mfk{sl}_2)$, on which $K\cdot v=-v$ and $Ev=E^{(p)} v=F v=F^{(p)} v=0$.  This special simple does not appear in $\rep (\SL_2)_q\subset \rep U_q(\mfk{sl}_2)$, as it is not graded by the character lattice.  Furthermore, we have
\[
\operatorname{Irrep}(\rep (\SL_2)_q)\cap \operatorname{Irrep}(\chi\ot \rep (\SL_2)_q)=\emptyset.
\]
One directly compares actions on highest weight vectors of simples, elaborated on in~\cite[Section 3.1]{bfgt09}, and employs the precise definition of $\msc{C}^+_p$ in~\cite[Section 3.4]{bfgt09}, to see that $\rep G_q=\msc{C}^+_p$ and $\chi\ot\rep G_q=\msc{C}^-_p$.  So we rephrase the conjecture of Bushlanov et al.

\begin{conjecture}[Bushlanov et al.~\cite{bfgt09}]
There is an equivalence of tensor categories $\rep (\SL_2)_q\overset{\sim}\to \rep \mcl{LM}(1,p)$.
\end{conjecture}

\subsection{Connecting some conjectures at $(1,p)$-central charge}

We consider the triplet vertex operator algebra $\mcl{W}_p$ and related singlet algebra $\mcl{M}_p$, with central charge $c_p$~\cite{kausch91,gaberdielkausch96,adamovicmilas07,adamovicmilas08}.  We have the sequence of vertex operator algebra extensions
\[
\mcl{L}_p\subset \mcl{M}_p\subset \mcl{W}_p.
\]
There is an integrable $\mfk{sl}_2$-action on $\mcl{W}_p$ by vertex derivations, and the $\mfk{h}$-weight spaces appearing in $\mcl{W}_p$ for this action are all even~\cite{adamoviclinmilas13,fgst06}.  Rather, we have a $\PSL_2=\SL_2^\vee$-action on $\mcl{W}_p$.  Under this $\PSL_2$-action we have
\[
\mcl{M}_p=\mcl{W}_p^{T^\vee}\ \ \text{and}\ \ \mcl{L}_p=\mcl{W}_p^{\PSL_2},
\]
where $T^\vee$ is the $1$-dimensional torus in $\PSL_2$~\cite[Eq.\ 5.8]{creutzigetal}.  Via this $\PSL_2$-action on $\mcl{W}_p$, we obtain a $\PSL_2$-action on $\rep \mcl{W}_p$ and may consider the equivariantizations $(\rep \mcl{W}_p)^{\PSL_2}$ and $(\rep\mcl{W})^{T^\vee}$, which are simply the categories of $\mcl{W}_p$-representations with compatible actions of $\PSL_2$ and $T^\vee$-respectively (or the associated Lie algebras if one prefers).  From this information we deduce the following.

\begin{lemma}\label{lem:1076}
Taking invariants provides $\mbb{C}$-linear functors
\[
\operatorname{A}:(\rep \mcl{W}_p)^{T^\vee}\to \rep \mcl{M}_p,\ V\mapsto V^{T^\vee},
\]
\[
\operatorname{B}:(\rep \mcl{W}_p)^{\PSL_2}\to \rep \mcl{L}_p,\ V\mapsto V^{\PSL_2}.
\]
\end{lemma}

In considering the following conjecture, one should compare the maps of Lemma~\ref{lem:1076} to the equivalence $(-)^R$ of Section~\ref{sect:eq_deq}.

\begin{conjecture}\label{conj:AB}
The functors $\operatorname{A}$ and $\operatorname{B}$ are fully faithful, $\operatorname{A}$ is an embedding, and $\operatorname{B}$ is an equivalence onto $\rep \mcl{LM}(1,p)\subset \rep \mcl{L}_p$.
\end{conjecture}

There is a rather vast network of conjectures regarding the algebras $\mcl{L}_p$, $\mcl{W}_p$, and $\mcl{M}_p$~\cite{gainutdinovetal06,bgt12,creutzigmilas14,cgp15}, of which we only recall a few.  For $\mcl{M}_p$, it is conjectured that some distinguished subcategory in $\rep \mcl{M}_p$ is a braided tensor category~\cite{creutzigmilas14,creutzigetal}.  It is also known that the category $\mcl{W}_p$ is a braided tensor category~\cite{adamovicmilas08,tsuchiyawood13}.  Furthermore, the $\PSL_2$-action on $\rep \mcl{W}_p$ should respect the braided tensor structure, so that the equivariantizations are also braided tensor categories.  So we conjecture further that map $\operatorname{A}$ is a braided tensor functors.  Furthermore, the image of $\operatorname{A}$ should be the centralizer of $\mcl{W}_p$ in the tensor subcategory $\rep_{\langle s\rangle} \mcl{M}_p$ generated by the simples~\cite[Conjecture 1.4]{creutzigetal}.
\par

We have a final conjecture which concerns the $\mbb{C}$-linear equivalences $f_p:\rep u_q^{\M}(\mfk{sl}_2)\to \rep \mcl{W}_p$ of~\cite{gainutdinovetal06,nagatomotsuchiya11}.

\begin{conjecture}\label{conj:1064}
The $\mbb{C}$-linear equivalence $f_p:\rep u_q^{\M}(\mfk{sl}_2)\to \rep \mcl{W}_p$ is $\PSL_2$-equivariant, or can be made to be so.
\end{conjecture}

This conjecture can seemingly ``just be checked".  However, the $\PSL_2$-action on $\rep u_q^{\M}(\mfk{sl}_2)$ is not so straightforward (see~\cite[\S 9.1]{negron}).  So, it may be preferable to first lift the equivalence $f_p$ to an equivalence from the canonical form
\[
F_p:(\rep (\SL_2)_q)_{\PSL_2}\overset{\sim}\to \rep \mcl{W}_p.
\]
At this level, the $\PSL_2$ action is fairly transparent on both sides.

\begin{proposition}[{cf.~\cite[Conjecture 1.4]{creutzigetal},~\cite{bfgt09}}]
Supposing Conjecture~\ref{conj:1064} is correct, then we have natural $\mbb{C}$-linear functors 
\[
\tilde{\operatorname{A}}:\rep\dotu_q(\mfk{sl}_2)\to \rep \mcl{M}_p\ \ \text{and}\ \  \tilde{\operatorname{B}}:\rep(\SL_2)_q\to \rep \mcl{L}_p.
\]
If furthermore Conjecture~\ref{conj:AB} holds, $\tilde{\operatorname{A}}$ is an embedding and $\tilde{\operatorname{B}}$ is an equivalence onto $\rep \mcl{LM}(1,p)$
\end{proposition}

\begin{proof}
One simply transports the invariants functors through the equivalences
\[
\rep\dotu_q(\mfk{sl}_2)\overset{\sim}\to (\rep u^{\M}_q(\mfk{sl}_2))^{T^\vee}\underset{\ref{conj:1064}}\cong(\rep \mcl{W}_p)^{T^\vee}
\]
\[
\text{and}\ \rep (\SL_2)_q\overset{\sim}\to (\rep u^{\M}_q(\mfk{sl}_2))^{\PSL_2}\underset{\ref{conj:1064}}\cong (\rep \mcl{W}_p)^{\PSL_2}
\]
of Proposition~\ref{prop:EqdEq}, Proposition~\ref{prop:894}, and Theorem~\ref{thm:ag}. 
\end{proof}

\appendix

\section{Details on rational (de-)equivariantization}
\label{sect:A}

We cover the details needed to prove Proposition~\ref{prop:885}.  As a first order of business let us provide the proof of Lemma~\ref{lem:1025}.

\begin{proof}[Proof of Lemma~\ref{lem:1025}]
The fact that any finitely presented object is compact follows from the fact that free objects $unit_\ast V$, for $V$ in $\msc{D}$, are compact, and left exactness of the $\Hom$ functor.  Now, for arbitrary $M$ in $\msc{D}_S$ we may write $M$ as the union $M=\varinjlim_\alpha M'_\alpha$ of its finitely generated submodules $M'_\alpha$.  For any finitely generated $M'$ we may write the kernel $N$ of a projection $unit_\ast V'=S\ot_\mbb{C} V'\to M'$ as a direct limit of finitely generated modules $N=\varinjlim_\beta N_\beta$ and hence write $M'$ as a direct limit of finitely presented modules $M'=\varinjlim_\beta M_\beta$, with $M_\beta=S\ot_\mbb{C} V'/N_\beta$.  Thus we may write arbitrary $M$ as a direct limit $M=\varinjlim_\kappa M_\kappa$ of finitely presented modules.  Compactness of $M$ implies that the identity factors through some finitely presented $M_\kappa$, and hence $M=M_\kappa$.
\end{proof}

\subsection{Equivariantization and the de-equivariantization}
\label{sect:eq_deq}

Suppose $F:\rep \Pi\to \msc{C}$ is a central embedding which is faithfully flat and locally finite.  Take
\[
R:=\O\text{ considered as a algebra object in $\rep\Pi$ with trivial $\Pi$-action}.
\]
We omit the prefix $F$ and write simply write $\O$ and $R$ for the images of these algebras in $\msc{C}$.  We define the functor on the de-equivariantization
\[
\psi_u:\msc{C}_\Pi\to (\msc{C}_\Pi)_R,\ \ \psi_u M:=R\ot M,
\]
where $\O$ acts diagonally on each $\psi_uM$ and $R$ acts via the first component.  More precisely, we have the algebra map $\Delta:\O\to R\ot\O$ in $\rep\Pi$ given by comultiplication and act naturally on $\psi_uM$ via $\Delta$.  For finite presentation, one observes on free modules $\O\ot V$ an easy isomorphism $\psi_u(\O\ot V)\cong unit_\ast(\O\ot V)$ in $(\msc{C}_\Pi)_R$, so that applying $\psi_u$ to a finite presentation for $M$, as an $\O$-module, yields a finite presentation for $\psi_uM$ over $R$.
\par

We have the natural iosmorphism
\[
\psi_u\psi_u(V)=R\ot (R\ot V)\cong(R\ot R)\ot V=\Delta_\ast \psi_u(V)
\]
given by the associativity in $\msc{C}$ and the natural isomorphism $\psi_u V\ot_{(R\ot \O)} \psi_u W\cong \psi_u(V\ot_\O W)$ given by multiplication from $R$.  Whence we have a canonical rational action of $\Pi$ on the de-equivariantization $\msc{C}_\Pi$, and can consider the corresponding equivariantization $(\msc{C}_\Pi)^\Pi$.  Objects in this category are simply $\O$-modules in $\msc{C}$ with a compatible $R$-coaction.
\par

Note that the $R$-coinvariants $X^R$ of an equivariant object $X$ is a $\msc{C}$-subobject in $X$, as it is the preimage of $\1\ot X\subset R\ot X$ under the $R$-coaction.  Whence we have the functor
\[
(-)^R:(\msc{C}_\Pi)^\Pi\to \Ind \msc{C},\ \ X\mapsto X^R.
\]
In addition, for any $V$ in $\msc{C}$ the object $\can^!(V)=\O\ot V$ can be given the $\O$-action and $R$-coaction from $\O$.  The coinvariants of $\can^!(V)$ is the subobject $\1\ot V$, and the unital structure on $\msc{C}$ provides a natural ismorphism $\zeta:(-)^R\circ\can^!\overset{\sim}\to id_\msc{C}$.  We also have the natural transformation $\gamma:\can^!\circ (-)^R\to id_{(\msc{C}_\Pi)^\Pi}$ given by the $\O$-action
\[
\gamma_X:\can^!(X^R)=\O\ot X^R\to X.
\]

\begin{lemma}
The transformation $\gamma$ is a natural isomorphism, and the coinvariants functor $(-)^R$ has image in $\msc{C}$.
\end{lemma}

\begin{proof}
We have the twisted comultiplication $\Delta^S:R\to \O\ot \O$, $f\mapsto f_1\ot S(f_2)$, and can define the inverse $\gamma_X^{-1}:X\to \O\ot X^R$ as the composite
\[
X\overset{\rho}\to R\ot X\overset{\Delta^S\ot 1}\to \O\ot \O\ot X\to \O\ot X,
\]
which one can check has image in $\O\ot X^R$ and does in fact provide the inverse to $\gamma$, just as in the Hopf case~\cite{montgomery93}.  To see that $X^R$ is in $\msc{C}$, and not in $\Ind\msc{C}\setminus\msc{C}$, we note that $X\cong \O\ot X^R$ is of finite length in $\msc{C}_\Pi$ and that $\O\ot-$ is exact, which forces $X^R$ to be of finite length.  Hence $X^R$ is in $\msc{C}$.
\end{proof}

Since both $\zeta$ and $\gamma$ are isomorphisms we have directly

\begin{proposition}[{cf.~\cite{arkhipovgaitsgory03,dgno10}}]\label{prop:EqdEq}
The functor $\can^!:\msc{C}\to (\msc{C}_\Pi)^\Pi$ is an equivalence of monoidal (and hence tensor) categories.
\end{proposition}

\begin{remark}
One can avoid all finiteness concerns by employing the $\Ind$-category $\Ind\msc{C}$ and the category of arbitrary modules $\O\text{-Mod}_{\Ind{\msc{C}}}$.  Then, with the cocomplete theory of Section~\ref{sect:cocomplete}, one can argue exactly as above to find that the functor $\can^!:\Ind\msc{C}\to (\O\text{-Mod}_{\Ind\msc{C}})^\Pi$ is again an equivalence.
\end{remark}

\subsection{De-equivariantizing the equivariantization}
\label{sect:deq_eq}

Let $\msc{D}$ be a tensor category equipped with a rational action of $\Pi$.  There is a canonical embedding $\rep\Pi\to \msc{D}^\Pi$ into the equivariantization which identifies $\rep\Pi$ with the preimage of $Vect\subset \msc{D}$ in $\msc{D}^\Pi$, under the forgetful functor.  Indeed, the fact that the action map $\psi_u:\msc{D}\to \msc{D}_R$ is monoidal implies that $\psi_u(\1)=R$, so that the restriction of $\psi_u$ to the trivial subcategory $Vect\subset \msc{D}$ is equated with the usual action of $\Pi$ on $Vect$, and hence $Vect^\Pi=\rep\Pi$.
\par

We have the two algebras $\O$ and $R$ in $\rep\Pi$, the latter one being trivial, which are equated under the composite $\rep\Pi\to \msc{D}^\Pi\to \msc{D}$, i.e.\ which are indistinguishable as objects in $\msc{D}$.  Hence the counit $\O\to \1$, which is not a map in $\rep\Pi$, is a map in $\msc{D}$, and for any $\O$-module in the equivariantization $\msc{D}^\Pi$ the reduction $X_\O:=\1\ot_\O X$ is a well-defined object in $\msc{D}$.
\par

Since $\O$ is trivial in $\msc{D}$, and $\psi_u$ is a tensor map, we have $\psi_u(\O)=R\ot\O$.  By the definition of $\O$ in $\rep\Pi$ the equivariant structure is given by the comultiplication $\Delta:\O\to R\ot\O$.  Hence $\O$ acts naturally on each $\psi_u(X)$ via the comultiplication, for any $\O=R$-module $X$ in $\msc{D}$.  So we can consider $\O$-modules in $\msc{D}^\Pi$ as $\msc{O}=R$-modules in $\msc{D}$ for which the coaction $X\to \psi_u(X)$ is $\O$-linear.
\par

For any object $V$ in $\msc{D}$ we consider $V$ as a trivial $\O$-module, and let $\O$ act on $\psi_u(V)$ diagonally.  Each $\psi_u(V)$ then becomes an object in $(\msc{D}^\Pi)_\Pi$ via the ``free" coaction, $\psi_u(V)\to \psi_u\psi_u(V)$ given by the unit of the $(\Delta_\ast,\Delta^\ast)$-adjunction
\[
\psi_u\overset{unit}\to \Delta^\ast\Delta_\ast\psi_u\overset{\Delta^\ast\sigma}\to \psi_u\psi_u.
\]
We have the reduction functor $1^\ast:(\msc{D}^\Pi)_\Pi\to \msc{D}$, $X\mapsto X_\O$, and the free functor $\can_!:\msc{D}\to (\msc{D}^\Pi)_\Pi$, $V\mapsto \psi_u(V)$.  There are natural transformations
\[
\eta_V:\psi_u(V)_\O=1^\ast\psi_u(V)\overset{\sim}\to V,\ \ \eta:1^\ast\circ \can_!\overset{\sim}\to id_{\msc{D}},
\]
and
\[
\vartheta_X:X\to \psi_u(X_\O),\ \ \vartheta:id_{(\msc{D}^\Pi)_\Pi}\to \can_!\circ 1^\ast,
\]
the former of which is simply given by the counit for $\psi_u$ and the latter is given as the composite $X\to \psi_u(X)\to \psi_u(X_\O)$ of the comultiplication and the application of $\psi_u$ to the reduction $X\to X_\O$ in $\msc{D}$.  The following is a consequence of the fact that each object in $(\msc{D}^\Pi)_\Pi$ is finitely presented over $\O$.

\begin{lemma}\label{lem:774}
The transformation $\vartheta$ is a natural isomorphism if and only if it is a natural isomorphism when applied to free modules $\O\ot W$, for $W$ in $\msc{D}^\Pi$.
\end{lemma}

\begin{lemma}\label{lem:778}
An object $X$ is $0$ in $(\msc{D}^\Pi)_\Pi$ if and only if the fiber $1^\ast X$ is $0$.
\end{lemma}

\begin{proof}
We may write $\msc{D}=\corep C$ for a coalgebra $C$, by Takeuchi reconstruction~\cite{takeuchi77}.  Then $\msc{D}_R$ is just the category of corepresentations of the $R$-coalgebra $C_R$ which are finitely presented over $R$.  Now, for a finitely presented $R$-module $M$ we understand that $M$ vanishes if and only if its fiber $x^\ast M$ vanishes for each closed point $x:\Spec(K)\to \Pi$.  Let $p(x):\O_K\to K$ be the corresponding ring map.  Note that the reduction simply takes the fiber at the identity.
\par

Take $M$ in $(\msc{D}^\Pi)_\Pi$ and suppose that $1^\ast M$ vanishes.  Consider a closed point $x\in \Pi(K)$.  By changing base to $\msc{D}_K$ and $\Pi_K$ we may assume that $K$ is our base field, so that $x^{-1}\cdot x=\epsilon$.  Via the the coaction we find an isomorphism
\begin{equation}\label{eq:787}
M\overset{\rho_M}\to \psi_u M\to p(x)_\ast \psi_u M=t_x M,
\end{equation}
where the last map is the counit of the $(p(x)_\ast,p(x)^\ast)$-adjunction, and $t:\Pi(K)\to \Aut(\msc{D})$ is the discrete action of $\Pi(K)$.
\par

Now, $t_x M$ has a canonical $\O=R$-action via the functorial identification $\End_\msc{D}(M)\cong \End_\msc{D}(t_xM)$, and the fiber $y^\ast M$ at a given $K$-point $y$ vanishes if and only if the fiber $y^\ast(t_x M)$ vanishes.  If we let $f_x:R\to R$ denote the automorphism given by left translation by $x$ then we see that~\eqref{eq:787} is an $R$-linear isomorphism from $M$ to the restriction of $t_xM$ along $f_x$.  In particular, we have
\[
0=1^\ast M\cong 1^\ast(\operatorname{res}_{f_x}t_xM)=x^\ast(t_x M),
\]
which implies $x^\ast M=0$.  Since $x$ was arbitrary, we see $M=0$ if $1^\ast M=0$.  Conversely, the fiber at the identity obviously vanishes if $M$ vanishes.
\end{proof}

\begin{proposition}\label{prop:dEqEq}
The functor $\can_!:\msc{D}\to (\msc{D}^\Pi)_\Pi$ is an equivalence of monoidal (and hence tensor) categories.  Furthermore, the embedding $F:\rep\Pi\to \msc{D}^\Pi$ is faithfully flat and locally finite
\end{proposition}

\begin{proof}
We prove that $\vartheta$ is an isomorphism on free modules.  Take $T=\O\ot V$ consider $\vartheta_T:T\to \psi_u(V)$.  We extend to a right exact sequence $T\to \psi_u(V)\to M\to 0$.  The counital property for $\psi_u$ implies that the fiber $1^\ast\vartheta$ is identified with the identity on $V$.  By right exactness of the reduction we have $1^\ast M=0$, and hence the cokernel vanishes by Lemma~\ref{lem:778}.
\par

We now extent $\vartheta_T$ to a left exact sequence $T'\overset{p}\to T\overset{\vartheta_T}\to \psi_u(V)\to 0$, with $p$ a map from a finite free module.  (We need to use the fact that $\psi_u(V)$ is finitely presented to verify that such an extension exists.)  Since $\psi_u$ is a monoidal functor it preserves duals~\cite[Exercise 2.10.6]{egno15}, it follows that $\psi_u(V)$ is dualizable in $\msc{D}_R$ with dual $\psi_u(V)^\vee\cong \psi_u(V^\ast)$.  Free modules $R\ot W$ are also dualizable with dual $R\ot W^\ast$.
\par

Note that $1^\ast:(\msc{D}_\Pi)^\Pi\to \msc{D}$ is a monoidal functor, and hence preserves duality as well, so that $1^\ast(\vartheta_T^\vee)$ is identified with the isomorphism $(1^\ast\vartheta_T)^\ast$.  So by the same arguments employed above the dual $\vartheta_T^\vee:\psi_u(V)^\vee\to T^\vee$ is also surjective.  Since the dual composite
\[
\psi_u(V)^\vee\to T^\vee\overset{p^\vee}\to (T')^\vee
\]
is $0$ we find that $p^\vee$ is $0$.  Since duality $(-)^\vee$ is an equivalence on the category of (left and right) dualizable objects in $(\msc{D}^\Pi)_\Pi$, it follows that $p=0$.  So $\vartheta_T$ is an isomorphism for each free $T$.  We now employ Lemma~\ref{lem:774} to find that $\can_!$ is an equivalence.  The fact that $\msc{D}$ is a tensor category and that $\can_!$ is an equivalence implies that $F$ is both faithfully flat and locally finite.
\end{proof}

\subsection{Proof of Proposition~\ref{prop:885}}
\label{sect:prop_proof}

\begin{proof}[Proof of Proposition~\ref{prop:885}]
Take $\msc{D}=\msc{C}_\Pi$.  We have the de-equivariantization functor $\msc{C}\to \msc{D}$.  For a sequence $\rep\Pi\to \msc{K}\overset{i}\to \msc{K}'\to \msc{C}$ we have the de-equivariantization $\msc{K}_\Pi\overset{i_\Pi}\to \msc{K}'_\Pi\to \msc{D}$, with $\msc{K}_\Pi$ and $\msc{K}'_\Pi$ stable under the action of $\Pi$.  By the definition of the equivalence of $\can^!$, in Section~\ref{sect:eq_deq}, we find that there is a diagram
\[
\xymatrix{
(\msc{K}_\Pi)^\Pi\ar[rr]^{(i_\Pi)^\Pi} & & (\msc{K}'_\Pi)^\Pi\\
\msc{K}\ar[u]^{\can^!}_\sim\ar[rr]^i & & \msc{K}'\ar[u]^{\can^!}_\sim.
}
\]
Hence $i$ is an equivalence if and only if $i_\Pi$ is an equivalence, and thus de-equivariantization $(-)_\Pi$ defines an inclusion of the poset of (isomorphism-closed) intermediate categories $\Pi\text{-Int}(\msc{C})=\{\rep\Pi\subset \msc{K}\subset \msc{C}\}$ to the poset $\Pi\text{-Stab}(\msc{D})=\{\msc{W}\subset \msc{D}\}$ of (isomorphism-closed) $\Pi$-stable categories.  A completely similar argument, using $\can_!$, shows that equivariantization $\msc{W}\subset \msc{D}\rightsquigarrow \msc{W}^\Pi\subset \msc{C}$ defines an inclusion of posets $\Pi\text{-Stab}(\msc{D})\to \Pi\text{-Int}(\msc{C})$ which is inverse to $(-)_\Pi$.
\par

Since de-/equivariantization under a central inclusion/braided action preserves braided subcategories, and central subcategories, the above argument shows that this bijection of posets restricts to a bijection for both braided and central subcategories as well.
\end{proof}

\bibliographystyle{abbrv}

\end{document}